%% file: Rank_reg_v2.tex
\documentclass{article}

\pdfoutput=1

\input{ex_shared_article.txt}

\newcommand{\TheTitle}{Low-Rank Inducing Norms with Optimality Interpretations}

\title{{\TheTitle} }

\author{Christian Grussler \thanks{Department of Engineering, Control Group, Cambridge University, United Kingdom \texttt{christian.grussler@eng.cam.ac.uk}.} \and Pontus Giselsson \thanks{Department of Automatic Control, Lund University, Sweden \texttt{pontusg@control.lth.se}.}}

\begin{document}
	
	\maketitle
	
	\input{sec_abstract.txt}

\input{sec_intro_v3.txt}
\input{sec_prelim.txt}

\input{sec_norms.txt}

\input{sec_main.txt}
\input{sec_comp.txt}	
\input{sec_ex_v2.txt}
\input{sec_conclusion.txt}	

\input{sec_appendix_v2.txt}

\bibliographystyle{plain}
\bibliography{refopt,refpos}

\end{document}

%% file: ex_shared_article.txt
\usepackage{lipsum}
\usepackage{amsfonts}
\usepackage{graphicx}
\usepackage{epstopdf}
\usepackage{algorithmic}
\usepackage{standalone}

\ifpdf
  \DeclareGraphicsExtensions{.eps,.pdf,.png,.jpg}
\else
  \DeclareGraphicsExtensions{.eps}
\fi

\usepackage{amsopn}

\usepackage{epsfig} 
\usepackage{times} 
\usepackage{amsmath} 
\usepackage{amssymb}  
\usepackage{mathptmx} 
\DeclareMathAlphabet{\mathcal}{OMS}{cmsy}{m}{n}
\usepackage[english]{babel}
\usepackage{url}
\usepackage{enumerate}
\usepackage{bbm}
\usepackage{caption}
\usepackage{algorithm}
\usepackage{color}
\usepackage{pgfplots}
\usepackage{siunitx}
\usepackage{tikz}
\usepackage{tkz-euclide}
\usepackage{chngcntr}
\usepackage[titletoc]{appendix}
\usepackage{enumerate}
\usepackage{verbatim}

\definecolor{ao(english)}{rgb}{0.0, 0.5, 0.0}
\usepackage[colorlinks, citecolor = {ao(english)}, linkcolor = {ao(english)}]{hyperref} 
\usepackage{cleveref}
\usepackage{aliascnt}
\usetikzlibrary{calc}
\usepackage{siunitx}
\usepackage{subcaption}
\usepackage{tikz}
\usepackage{pgfplots}
\usetikzlibrary{arrows,shapes,trees,calc,positioning,patterns,decorations.pathmorphing,decorations.markings}
\usetikzlibrary{matrix}
\usepackage{caption}
\usepgfplotslibrary{groupplots}
\pgfplotsset{compat=newest}

\newtheorem{thm}{Theorem}
\crefname{thm}{Theorem}{Theorems}
\newtheorem{prop}{Proposition}
\crefname{prop}{Proposition}{Propositions}
\newtheorem{lem}{Lemma}
\crefname{lem}{Lemma}{Lemmas}
\newtheorem{cor}{Corollary}
\crefname{cor}{Corollary}{Corollaries}
\newtheorem{rem}{Remark}
\crefname{rem}{Remark}{Remark}

\crefname{ass}{Assumption}{Assumption}
\usepackage{dsfont}
\let\mathbb=\mathds
\usepackage[numbers,sort&compress]{natbib}

\crefname{conj}{Conjecture}{Conjectures}
\newtheorem{defn}{Definition}
\crefname{defn}{Definition}{Definitions}

\crefname{prob}{Problem}{Problems}
\crefname{algorithm}{Algorithm}{Algorithms}
\crefname{Figure}{Fig.}{Figs.}
\newenvironment{proof}{\smallbreak\noindent{\it Proof. }}{\hfill$\Box$\smallbreak}

\newcommand{\Rmn}{\mathbb{R}^{n \times m}}

\newcommand{\Rnn}{\mathbb{R}^{n \times n}}

\newcommand{\conv}{\textnormal{conv}}
\newcommand{\rk}{\textnormal{rank}}
\newcommand{\svd}{\textnormal{svd}}
\newcommand{\diag}{\textnormal{diag}}

\newcommand{\argmin}{\operatornamewithlimits{argmin}}

\newcommand{\trace}{\textnormal{trace}}

\newcommand{\minmn}{\min \{ m,n \}}

\newcommand{\cl}{\textnormal{cl}}

\newcommand{\prox}{\textnormal{prox}}

\newcommand{\opts}{\star}
\newcommand{\card}{\textnormal{card}}
\newcommand{\dom}{\textnormal{dom}}

\newcommand{\epi}{\textnormal{epi}}

\newcommand{\normrgast}[1]{\left\| #1\right \|_{g,r\ast}}

\newcommand{\normkyr}[1]{\left\| #1\right \|_{\ell_1,r}}

\newcommand{\normkyrast}[1]{\left\| #1\right \|_{\ell_{\infty},r\ast}}

\newcommand{\normA}[2]{\left\| #1\right \|_{ #2}}
\newcommand{\normAast}[2]{\left\| #1\right \|_{ #2\ast}}
\newcommand{\funcdom}{\Rmn \to \mathbb{R} \cup \lbrace \infty \rbrace}

\newcommand{\transp}{\mathsf{T}}

\colorlet{FigColor1}{blue}
\colorlet{FigColor2}{red}
\colorlet{FigColor3}{green}
\colorlet{FigColor4}{orange}
\pgfplotsset{every axis plot/.append style={line width=1.5pt}}

\crefformat{equation}{\textup{#2(#1)#3}}
\crefrangeformat{equation}{\textup{#3(#1)#4--#5(#2)#6}}
\crefmultiformat{equation}{\textup{#2(#1)#3}}{ and \textup{#2(#1)#3}}
{, \textup{#2(#1)#3}}{, and \textup{#2(#1)#3}}
\crefrangemultiformat{equation}{\textup{#3(#1)#4--#5(#2)#6}}%
{ and \textup{#3(#1)#4--#5(#2)#6}}{, \textup{#3(#1)#4--#5(#2)#6}}{, and \textup{#3(#1)#4--#5(#2)#6}}

\Crefformat{equation}{#2Equation~\textup{(#1)}#3}
\Crefrangeformat{equation}{Equations~\textup{#3(#1)#4--#5(#2)#6}}
\Crefmultiformat{equation}{Equations~\textup{#2(#1)#3}}{ and \textup{#2(#1)#3}}
{, \textup{#2(#1)#3}}{, and \textup{#2(#1)#3}}
\Crefrangemultiformat{equation}{Equations~\textup{#3(#1)#4--#5(#2)#6}}%
{ and \textup{#3(#1)#4--#5(#2)#6}}{, \textup{#3(#1)#4--#5(#2)#6}}{, and \textup{#3(#1)#4--#5(#2)#6}}

\crefdefaultlabelformat{#2\textup{#1}#3}

%% file: sec_abstract.txt
\begin{abstract}
Optimization problems with rank constraints appear in many diverse fields such as control, machine learning and image analysis. Since the rank constraint is non-convex, these problems are often approximately solved via convex relaxations. Nuclear norm regularization is the prevailing convexifying technique for dealing with these types of problem. This paper introduces a family of low-rank inducing norms and regularizers which includes the nuclear norm as a special case. A posteriori guarantees on solving an underlying rank constrained optimization problem with these convex relaxations are provided. We evaluate the performance of the low-rank inducing norms on three matrix completion problems. In all examples, the nuclear norm heuristic is outperformed by convex relaxations based on other low-rank inducing norms. For two of the problems there exist low-rank inducing norms that succeed in recovering the partially unknown matrix, while the nuclear norm fails. These low-rank inducing norms are shown to be representable as semi-definite programs. Moreover, these norms have cheaply computable proximal mappings, which makes it possible to also solve problems of large size using first-order methods. 
\end{abstract}

%% file: sec_intro_v3.txt
\section{Introduction}
Many problems in machine learning, image analysis, model order reduction, multivariate linear regression, etc.
(see~\cite{izenman1975reduced,antoulas2005approximation,candes2009exact,candes2010matrix,recht2010guaranteed,chandrasekaran2011rank,velu2013multivariate,hastie2015statistical,larsson2016convex,vidal2016generalized}), can be posed as a low-rank estimation problems based on measurements and prior information about a data matrix. These estimation problems often take the form
\begin{equation}
\begin{aligned}
& \underset{M}{\textnormal{minimize}}
& & f_0(M)\\
& \textnormal{subject to}
& & \rk(M) \leq r,
\end{aligned}
\label{eq:rank_prob_intro}
\end{equation}
where $f_0$ is a proper closed convex function and $r$ is a positive integer that specifies the desired or expected rank. Due to non-convexity of the rank constraint, a solution to \cref{eq:rank_prob_intro} is known in only a few special cases (see~e.g.~\cite{antoulas1997approximation,antoulas2005approximation,velu2013multivariate}). 

A common approach to deal with the rank constraint is to use the nuclear norm heuristic (see~\cite{fazel2001rank,recht2010guaranteed}). The idea is to convexify the problem by replacing the non-convex rank constraint with a nuclear norm regularization term. For matrix completion problems, where some elements of a low-rank matrix are known and the objective is to recover the unknown entries, the nuclear approach is known to recover the true low-rank matrix with high probability, provided that enough random measurements are available (see~\cite{candes2009exact,recht2010guaranteed,chandrasekaran2012convex} and \Cref{sec:compl}). If this assumption is not met, however, the nuclear norm heuristic may fail in producing satisfactory estimates (see~\cite{grussler2016low,grussler2016covariance}).


This paper introduces a family of low-rank inducing norms as alternatives to the nuclear norm. These norms can be interpreted as the largest convex minorizers (convex envelopes) of non-convex functions $f$ of the form
\begin{align}
f = \|\cdot\|+\chi_{\rk(\cdot)\leq r},
\label{eq:f_reg_intro}
\end{align}
where $\|\cdot\|$ is an arbitrary unitarily invariant norm, and $\chi_{\rk(\cdot)\leq r}$ is the indicator function for matrices with rank less than or equal to $r$. This interpretation motivates the use of low-rank inducing norms in convex relaxations to \cref{eq:rank_prob_intro}, especially if $f_0$ in \cref{eq:rank_prob_intro} can be split into the sum of a convex function and a unitarily invariant norm. If the solution to the convex relaxation has rank $r$, then the convex envelope and the original non-convex function coincide at the optimal point, and the convex relaxation solves the original non-convex problem~\cref{eq:rank_prob_intro}. Therefore, our interpretation provides easily verifiable \emph{a posteriori} optimality guarantees for solving~\cref{eq:rank_prob_intro} using convex relaxations based on low-rank inducing norms. We also show that the nuclear norm is obtained as the convex envelope of \cref{eq:f_reg_intro} with target rank $r=1$. This yields the novel interpretation that the nuclear norm is the convex envelope of a unitarily invariant norm restricted to rank-1 matrices. This interpretation is different to the convex envelope interpretation in~\cite{fazel2001rank}.

Within the family of low-rank inducing norms, there is a flexibility in the choice of target rank $r$ as well as the unitarily invariant norm in \cref{eq:f_reg_intro}. This work considers particularly the Frobenius- and spectral norm. We refer to the corresponding low-rank inducing norms as \emph{low-rank inducing Frobenius norms} and \emph{low-rank inducing spectral norms}, respectively. Which norm to choose may depend on the desired properties of the solution to the non-convex problem~\cref{eq:rank_prob_intro}. For example, assume that \cref{eq:rank_prob_intro} is a matrix completion problem and suppose that the unknown elements are expected to be of small magnitude compared to the known ones. Then the low-rank inducing Frobenius norm may be a suitable choice, since it penalizes the magnitude of all the unknown elements. On the other hand, if the magnitudes of the unknown entries are of comparable size to the known ones, then the low-rank inducing spectral norm might be preferable. In \Cref{sec:compl}, two examples are provided to illustrate these differences. Further modelling guidelines are outside the scope of this paper. 

However, in cases such as the Frobenius norm approximation problem, it is evident from the objective function of the non-convex problem \cref{eq:rank_prob_intro} that the low-rank inducing Frobenius norms should be used (see~\Cref{sec:main}). These norms, also called $r\ast$ norms, have been previously discussed in the literature (see \cite{bach2012optimization,eriksson2015k,McDonaldPS15,grussler2015optimal,doan2016finding,grussler2016low,grussler2016covariance}). In~\cite{bach2012optimization,eriksson2015k,McDonaldPS15,doan2016finding}, no optimality interpretations are provided, but in previous work we have presented such interpretations for the squared $r\ast$ norms (see~\cite{grussler2015optimal,grussler2016low,grussler2016covariance}). In this paper, these results are generalized to hold for all convex increasing functions of general low-rank inducing norms. Most importantly, our results hold for linear increasing functions, i.e. for the low-rank inducing norm itself. These generalizations require different proof techniques than the ones used in~\cite{grussler2015optimal,grussler2016low,grussler2016covariance}. To the best of our knowledge, no other low-rank inducing norms from the proposed family, including low-rank inducing spectral norms, have been proposed in the literature.



For the family of low-rank inducing norms to be useful in practice, they must be suitable for numerical optimization. We show that the low-rank inducing Frobenius- and spectral norms are representable as semi-definite programs (SDP). This allows us to readily formulate and solve small to medium scale problems using standard SDP-solvers (see~\cite{peaucelle2002user,toh2004implementation}). Moreover, these norms can be shown to have cheaply computable proximal mappings, comparable with the computational cost for the proximal mapping of the nuclear norm. Therefore, it is possible to solve large-scale problems involving low-rank inducing norms by means of proximal splitting methods (see~\cite{combettes2011proximal,parikh2014proximal}). 

Besides the two matrix completion problems mentioned above, the performance of different low-rank inducing norms is evaluated on a covariance completion problem. The evaluation reveals that the choice of low-rank inducing norms has tremendous impact on the ability to complete the covariance matrix. In particular, the nuclear norm is significantly outperformed by the low-rank inducing Frobenius norm, as well as the low-rank inducing spectral norm. Implementations to this work are available at \cite{grussler2018github}.


This paper is organized as follows. We start by introducing some preliminaries in~\Cref{sec:prelim}. In~\Cref{sec:norms}, we introduce the class of low-rank inducing norms, and provide optimality interpretations of these in~\Cref{sec:main}. In~\Cref{sec:comp}, computability of low-rank inducing Frobenius- and spectral norms is addressed. To support the usefulness of having more low-rank inducing norms at our supply, numerical examples are presented in~\Cref{sec:compl}.

%% file: sec_prelim.txt
\section{Preliminaries}
\label{sec:prelim}
The set of reals is denoted by $\mathbb{R}$, the set of real vectors by $\mathbb{R}^n$, and the set of real matrices by $\Rmn$. Element-wise nonnegative matrices $X\in\Rmn$ are denoted by \linebreak $X \in \mathbb{R}^{n \times m}_{\geq 0}$. If symmetric $X\in\mathbb{R}^{n\times n}$ is positive definite (semi-definite), we write $X \succ 0$ ($X \succeq 0$). These notations are also used to describe relations between matrices, e.g. $A \succeq B$ means $A-B \succeq 0$. The non-increasingly ordered singular values of $X \in \Rmn$, counted with multiplicity, are denoted by $\sigma_1(X) \geq \dots \geq \sigma_{\min \{ m,n \}}(X)$.  Furthermore, $$\langle X , Y \rangle := \sum_{i=1}^{m}\sum_{j=n}^{n} x_{ij} y_{ij} = \trace(X^\transp Y)$$ defines the Frobenius inner-product for $X,Y \in \Rmn$, which gives the \emph{Frobenius~norm} $$\|X\|_F := \sqrt{\trace(X^\transp X)}=\sqrt{\sum_{i=1}^{n}\sum_{j=1}^{m} x_{ij}^2} =  \sqrt{\sum_{i=1}^{\minmn} \sigma_i^2(X)}.$$
The Frobenius norm is unitarily invariant, i.e.  $\|UXV\|_F = \|X\|_F$ for all unitary matrices $U\in \mathbb{R}^{n \times n}$ and $V \in \mathbb{R}^{m \times m}$. We define for all $x=(x_1,\ldots,x_q)\in\mathbb{R}^q$ all $q \in \mathbb{N}$ the norms
\begin{align}
\ell_{1}(x)&:={\sum_{i=1}^{q} |x_i|}
\label{eq:ell_defs}, &\ell_2(x)&:=\sqrt{\sum_{i=1}^{q} x_i^2},&\ell_{\infty}(x)&:=\max_{i}|x_i|.
\end{align}
Then the Frobenius norm satisfies $\|X\|_F=\ell_2(\sigma(X))$, where $$\sigma(X):=(\sigma_1(X),\ldots,\sigma_{\minmn}(X)).$$ The functions $\ell_1$ and $\ell_{\infty}$ define the nuclear norm $\|X\|_{\ell_1}:=\ell_{1}(\sigma(X))$ and the spectral norm $\|X\|_{\ell_\infty}:=\ell_{\infty}(\sigma(X))=\sigma_1(X)$. 

For a set $\mathcal{C} \subset \Rmn$,
\begin{align*}
\chi_{\mathcal{C}}(X) := \begin{cases}
	0, & X \in \mathcal{C}\\
	\infty, & X\notin \mathcal{C}
\end{cases}
\end{align*}
denotes the so-called \emph{indicator function}. We also use $\chi_{\rk(\cdot)\leq r}$ to denote the indicator function of the set of matrices which have at most rank $r$. 

The following function properties will be used in this paper. The \emph{effective domain} of a function $f: \funcdom$ is defined as $$\dom (f) :=  \lbrace X \in \Rmn: f(X) < \infty \rbrace$$ and the \emph{epigraph} is defined as $$\epi(f) := \lbrace (X,t) : f(X) \leq t, X \in \dom (f), \ t \in \mathbb{R} \rbrace.$$ Further, $f$ is said to be: 
			 \begin{itemize}
			 	\item \emph{proper} if $\dom(f) \neq \emptyset$.
			 	\item \emph{closed} if $\epi(f)$ is a closed set.
			 	\item \emph{positively homogeneous (of degree 1)} if for all $X \in \dom(f)$ and $t > 0$ it holds that $f(tX) = tf(X)$.
			 	\item \emph{nonnegative} if $f(X) \geq 0$ for all $X \in \dom(f).$
			 	\item \emph{coercive} if $\lim_{\|X\|_F \to \infty} f(X) = \infty.$
			 \end{itemize}
A function $f:\mathbb{R}\cup\lbrace\infty\rbrace\to\mathbb{R}\cup\lbrace\infty\rbrace$ is called \emph{increasing} if $$x \leq y \ \Rightarrow \ f(x) \leq f(y) \ \text{ for all } \ x,y \in \dom(f)$$ and if there exist $x,y\in\mathbb{R}$ such that $x<y$ and $f(x)<f(y)$.

The \emph{conjugate (dual) function} $f^\ast$ of $f$ is defined as $$f^\ast(Y) :=   \sup_{X \in \Rmn} \left[ \langle X, Y \rangle - f(X) \right]$$
for all $Y \in \Rmn$. As long as $f$ is proper and minorized by an affine function, the conjugate $f^\ast$ is proper, closed and convex (see~\cite{hiriart2013convex}). The function $f^{\ast \ast} := (f^\ast)^\ast$ is called the \emph{biconjugate function} of $f$ and can be shown to be a convex minorizer of $f$, i.e.
\begin{equation*}
f(X) \geq f^{\ast \ast}(X) \ \text{for all} \ X \in \Rmn.
\end{equation*}
In fact, $f^{\ast \ast}$ is the point-wise supremum of all affine functions majorized by $f$ and therefore the largest convex minorizer of $f$. {{Commonly, $f^{\ast\ast}$ is also referred to as the convex envelope of $f$, which is motivated by the following equivalent characterization}} (see~\cite[Theorem~X.1.3.5, Corollary~X.1.3.6]{hiriart1996convex2}).
	\begin{lem}
		\label{lem:biconj_closed}
		Let $f:\funcdom$ be such that $f^{\ast \ast}$ is proper. Then 
		\begin{align*}
		\epi(f^{\ast \ast}) = \cl(\conv(\epi (f))), 
		\end{align*}
		where $\cl(\cdot)$ denotes the topological closure of a set and $\conv(\cdot)$ the convex hull. Further, $f^{\ast \ast} = f$ if and only if $f$ is proper, closed and convex.
	\end{lem} 
\cref{lem:biconj_closed} implies that for a proper, but possibly non-convex function $f$, it holds that
$$\inf_{X \in \Rmn} f(X) = \inf_{X \in \Rmn} f^{\ast \ast}(X)$$
and therefore $f^{\ast \ast}$ is the largest convex minorizer of $f$. However, determining $f^{\ast \ast}$ is as difficult as minimizing the non-convex function $f$. Instead, it is common to convexify the problem by splitting the function into $f = f_1 + f_2$, such that $f_1^{\ast \ast}$ and $f_2^{\ast \ast}$ can be easily computed. If $f_1$ is proper, closed and convex, then $f_1=f_1^{**}$ and $f_1 + f_2^{\ast \ast}$ is the largest convex minorizer of $f$ that keeps $f_1$ as a summand. In general, $ f^{\ast \ast} \neq f_1 + f_2^{\ast \ast}$, i.e. $f_1 + f_2^{\ast \ast}$ is not the largest convex minorizer of $f$. Therefore, the following lemma follows.
\begin{lem}
	\label{lem:zero-dual}
Let $f_1, f_2: \funcdom$ be such that $f_1$ is proper, closed and convex. Then,
\begin{align}
\inf_{X \in \Rmn} \left[ f_1(X) +f_2(X) \right] \geq \inf_{X \in \Rmn}\left[f_1(X)+f_2^{\ast \ast}(X)\right]. \label{lem:fenchel_inequ}
\end{align}	
If $X^\opts$ is a solution to the right-hand side and $f_2(X^\opts) = f_2^{\ast \ast}(X^\opts)$, then equality holds and $X^\opts$ is also a solution to the left-hand side.
\end{lem}
\Cref{lem:zero-dual} motivates the use of our terminology that $f_1 + f_2^{\ast \ast}$ is the \emph{optimal convex relaxation} of a given splitting $f_1 + f_2$, when $f_1$ is proper, closed and convex. 

Finally, if $f: \mathbb{R} \to \mathbb{R} \cup \{ \infty \}$, then the \emph{monotone conjugate} is defined as $$f^+(y) := \sup_{x \geq 0} \left[ \langle x, y \rangle - f(x) \right] \ \text{ for all } y\in\mathbb{R}.$$

%% file: sec_norms.txt
\section{Low-Rank Inducing Norms}
\label{sec:norms}

This section introduces the family of \emph{low-rank inducing norms}, which includes the nuclear norm as a special case. These norms can be used as regularizers in optimization problems to promote low-rank solutions. To define them, we need to characterize the class of unitarily invariant norms in terms of symmetric gauge functions. This characterization can be found in, e.g.~\cite[Theorem~7.4.7.2]{horn2012matrix}.
			\begin{defn}
		\label{def:gauge}
		A function $g: \mathbb{R}^q \to \mathbb{R}_{\geq 0}$ is a symmetric gauge function if
		\begin{enumerate}[i.]
			\item $g$ is a~norm.
			\item $\forall x \in \mathbb{R}^{q}: g(|x|) = g(x)$, where $|x|$ denotes the element-wise absolute value.
			\item $g(Px) = g(x)$ for all permutation matrices $P\in\mathbb{R}^{q\times q}$ and all $x\in\mathbb{R}^q$.
		\end{enumerate}
	\end{defn}
	
	\begin{prop}
		\label{thm:symmgauge}
		The norm $\| \cdot \|:\Rmn\to\mathbb{R}$ is unitarily invariant if and only if $$\| X \| = g(\sigma_1(X),\dots,\sigma_{\min \lbrace m,n \rbrace }(X))$$ for all $X\in\Rmn$, where $g$ is a symmetric gauge function.
	\end{prop}
As noted in~\Cref{sec:prelim}, the gauge functions for the Frobenius norm, spectral norm, and nuclear norm are $g=\ell_{2}$, $g=\ell_{\infty}$, and $g=\ell_{1}$, respectively, where $\ell_1$, $\ell_2$, and $\ell_{\infty}$ are defined in \cref{eq:ell_defs}.

The dual norm of a unitarily invariant norm is also unitarily invariant (see~\cite[Theorem~5.6.39]{horn2012matrix}). Therefore, it has an associated symmetric gauge function. This will be denoted by $g^D$, if the symmetric gauge function of the original norm is denoted by $g$. More specifically, let $M \in \Rmn$, $q:= \minmn$, and $g: \mathbb{R}^q \to \mathbb{R}_{\geq 0}$ be a symmetric gauge function associated with a unitarily invariant norm
	\begin{align*}
	\|M\|_g &:= g(\sigma_1(M),\dots,\sigma_q(M)).
	\end{align*} 
Then the dual of this norm is defined as
	\begin{align}
	\|Y\|_{g^D} &:= \max_{\|M\|_g \leq 1} \langle Y, M \rangle= g^D(\sigma_1(Y),\dots,\sigma_q(Y)) ,
\label{eq:dual_uni_norm}
	\end{align} 
	where the dual gauge function $g^D$ satisfies
	\begin{align}
	g^D(\sigma_1(Y),\dots,\sigma_q(Y)) = \max_{g(\sigma_1(M),\dots,\sigma_q(M)) \leq 1} \sum_{i=1}^{q} \sigma_i(M)\sigma_i(Y).
\label{eq:gD_def}
	\end{align}
The low-rank inducing norms will be defined as the dual norm of a rank constrained dual norm. This rank constrained dual norm is defined as
\begin{equation}
\| Y \|_{g^D,r} :=  \max_{\stackrel{\rk(M) \leq r}{\|M\|_g \leq 1}} \langle M, Y \rangle \label{eq:def_rank_ind_norm_dual}
\end{equation}
and the corresponding \emph{low-rank inducing norm} as
\begin{align}
\|M\|_{g,r*} := \max_{\|Y\|_{g^D,r}\leq 1}\langle Y,M\rangle.
\label{eq:def_rank_ind_norm}
\end{align}
For {{$r=\min\{m,n\}$}}, the rank constraint in \cref{eq:def_rank_ind_norm_dual} is redundant and the dual of the dual becomes the norm itself. 

For symmetric gauge functions $g:\mathbb{R}^q\to\mathbb{R}_{\geq 0}$, we denote their truncated symmetric gauge functions by $g(\sigma_1,\dots,\sigma_r) := g(\sigma_1,\dots,\sigma_r,0,\dots,0)$ for any $r \in \{1,\ldots,q\}$. With this notation in mind, some properties of low-rank inducing norms and their duals are stated in the following lemma. A proof is given in~\Cref{app:lem_norm_pf}. 
	\begin{lem}
		\label{lem:norm}
		Let $M,Y \in \Rmn$, $r \in \mathbb{N}$ be such that $1\leq r \leq q : = \minmn$, and $g: \mathbb{R}^q \to \mathbb{R}_{\geq 0}$ be a symmetric gauge function. Then $\|\cdot\|_{g^D,r}$ is a unitarily invariant norm that satisfies
		\begin{equation}
\| Y \|_{g^D,r} = g^D(\sigma_1(Y),\dots,\sigma_r(Y)).
\label{eq:firstass}
		\end{equation} 
		Its dual~norm $\|\cdot\|_{g,r\ast}$ satisfies
		\begin{align} 
		\| M \|_{g,r\ast} &=\max_{g^D(\sigma_1(Y),\dots,\sigma_r(Y)) \leq 1} \left[ \sum_{i=1}^{r} \sigma_i(M) \sigma_i(Y) + \sigma_r(Y) \sum_{i=r+1}^{q} \sigma_i(M) \right], \label{eq:grast_explic}
		\end{align} 
		and
		\begin{align}
		&\|M\|_g = \|M\|_{g,q\ast} \leq \dots \leq \|M\|_{g,1\ast}, \label{eq:norm_ineq}\\
		&\rk(M)  \leq r \ \Rightarrow \ \|M \|_{g} = \|M \|_{g,r\ast}. \label{eq:rank_norm}
		\end{align}
	\end{lem} 
This paper particularly focuses on low-rank inducing norms originating from the Frobenius norm and the spectral norm. When the original norm is the Frobenius norm, then $g=\ell_2$. Since the norm is self dual, it satisfies $g^D=\ell_2^D=\ell_2$. Then the truncated version in \cref{eq:firstass} becomes
	\begin{align*}
	\|Y\|_{\ell_2^D,r}:=\sqrt{\sum_{i=1}^r \sigma_i^2(Y)}.
	\end{align*}
The corresponding low-rank inducing norm is referred to as the \emph{low-rank inducing Frobenius norm}, and is denoted by
	\begin{align*}
	\|M\|_{\ell_2,r*:}=\max_{\|Y\|_r \leq 1} \langle Y,M \rangle.
	\end{align*}
If the original norm, instead, is the spectral norm, we have $g=\ell_{\infty}$. The dual norm is the nuclear (trace) norm (see \cite[Theorem~5.6.42]{horn2012matrix}), with gauge function $g^D=\ell_1$. Its truncated version is given by
	\begin{align*}
	\normkyr{Y} &:= \sum_{i=1}^r \sigma_i(Y),
	\end{align*}
and its dual, which we refer to as the \emph{low-rank inducing spectral norm}, is denoted by
	\begin{align*}
	\normkyrast{M}:=\max_{\normkyr{Y}\leq 1}\langle Y,M\rangle.
	\end{align*}
The nuclear norm is a special case of these low-rank inducing norms, corresponding to $r=1$.
\begin{prop}
The nuclear norm satisfies $\|\cdot\|_{\ell_1}=\|\cdot\|_{g,1*}$, where $\|\cdot\|_g$ is any unitarily invariant norm with $g(\sigma_1)=|\sigma_1|$.
\label{prop:nuclear_norm_characterization}
\end{prop}
A proof to this proposition is found in \Cref{app:prop_nuclear_norm_characterization_pf}.

Next, we state a result that is the key to our optimality interpretations for low-rank inducing norms in the next section. 

\begin{lem}\label{lem:convhull}
Let $B_{g,r\ast}^{1} := \{X \in \Rmn : \|X\|_{g,r\ast} \leq 1 \}$ be the unit low-rank inducing norm ball and let 
\begin{align}
E_{g,r} := \{X \in \Rmn : \|X\|_{g} = 1, \ \rk(X) \leq r \}.
\label{eq:def_Er}
\end{align}
Then $B_{g,r\ast}^{1} = \conv(E_{g,r})$, i.e. all $M\in\Rmn$ can be decomposed as
\begin{equation*}
\textstyle M = \sum_i \alpha_i M_i \quad \text{with} \quad \sum_i \alpha_i = 1, \ \alpha_i \geq 0,
\end{equation*}
where $M_i$ satisfies $\rk(M_i) \leq r$ and 
\begin{equation*}
\|M_i\|_{g} = \|M_i\|_{g,r\ast} = \|M\|_{g,r\ast}.
\end{equation*}
\end{lem}
A proof to this lemma is given in~\Cref{app:lem_convhull_pf}. The result is a direct consequence of \cref{lem:norm} and extends {{the well-known result for the unit-ball of the nuclear norm, which can be parametrized as the convex hull of all rank-1 matrices with unit spectral (or Frobenius) norm (see~\cite{watson1992characterization}).}}

In many cases, the set $E_{g,r}$ is the set of extreme points to the unit ball $B_{g,r\ast}^1$. The following result is proven in~\Cref{app:prp_rank_r_ind_str_conv_pf}.
\begin{prop}
Suppose that $\|\cdot\|_g$ satisfies
\begin{equation*}
 \textstyle \|\sum_{i}\alpha_i M_i\|_g<\sum_i\alpha_i \|M_i\|_g
\end{equation*}
for all $\alpha_i\in(0,1)$ such that $\sum_i\alpha_i=1$, and all {{distinct}} $M_i\in\Rmn$ with $\|M_i\|_g=1$.  Then $E_{g,r}$ in \cref{eq:def_Er} is the set of extreme points to $B_{g,r*}^1$.
\label{prp:rank_r_ind_str_conv}
\end{prop}
All $\ell_p$ norms with $1<p<\infty$ satisfy these assumptions, and therefore the unit balls of their low-rank inducing norms have $E_{g,r}$ as their extreme point sets. The extreme point sets for the unit balls of the low-rank inducing spectral norms are characterized next.
\begin{cor}\label{prp:rank_r_ind_spectral}
The extreme point set of the unit ball to the low-rank inducing spectral norm $B_{\ell_{\infty},r*}^1$ is given by
\begin{equation*}
\mathcal{E}_r := \{X\in\Rmn: \sigma_1(X)=\cdots=\sigma_r(X)=1  \ \text{ and } \ \rk(X) = r\}.
\end{equation*}
\end{cor}
This result is proven in \Cref{app:prp_rank_r_ind_spectral}.

We could also use the nuclear norm as a basis for the low-rank inducing norm. By \cref{prop:nuclear_norm_characterization}, we know that $\|\cdot\|_{\ell_1,1*}=\|\cdot\|_{\ell_1}$. Therefore \cref{eq:norm_ineq} implies that any low-rank inducing nuclear norm is just the nuclear norm, i.e. $$\|\cdot\|_{\ell_1}=\|\cdot\|_{\ell_1,q*}=\cdots=\|\cdot\|_{\ell_1,1*}.$$ Compared to using the low-rank inducing Frobenius- and spectral norms, this does not provide us with a richer family of low-rank inducing norms.


%% file: sec_main.txt
\section{Optimality Interpretations}
	\label{sec:main}
In this section, we show that low-rank inducing norms can be interpreted as {{convex envelopes}}, i.e. biconjugates of non-convex functions of the form \cref{eq:f_reg_intro}, where the norm is arbitrary but unitarily invariant. Using this interpretation, it is demonstrated how optimal convex relaxations of rank constrained optimization problems can be obtained. This yields a posteriori guarantees on when a convex relaxation involving a low-rank inducing norm solves the corresponding rank constrained problem. 

The interpretation of low-rank inducing norms follows as a special case of the following more general result. 

	\begin{thm}
		\label{thm:main}
		Assume $f: \mathbb{R}_{\geq 0} \to \mathbb{R} \cup \{ \infty \}$ is an increasing closed  convex function, and let $f_{\textnormal{reg}} := f(\|\cdot\|_g) + \chi_{\rk(\cdot)  \leq r}$ with $r \in \mathbb{N}$ such that $1 \leq r \leq \minmn$. Then,
		\begin{align}
			f^{\ast}_{\textnormal{reg}} &= f^{+}(\|\cdot\|_{g^D,r}), \label{eq:conj_reg}\\
		f^{\ast \ast}_{\textnormal{reg}} &= f(\|\cdot\|_{g,r\ast}). \label{eq:biconj_reg} 
		\end{align}
	\end{thm}
	\begin{proof}
Since $\epi(f(\|\cdot\|_{g,r\ast}))$ is closed by~\cite[Proposition~IV.2.1.8]{hiriart2013convex}, it follows by \cref{lem:biconj_closed} that if
		\begin{align*}
		\epi(f(\|\cdot\|_{g,r\ast})) = \conv( \epi(f_{\textnormal{reg}})),
		\end{align*}
		then \cref{eq:biconj_reg} follows.

Let us start by showing that $\epi(f(\|\cdot\|_{g,r\ast})) \subset \conv( \epi(f_{\textnormal{reg}})).$ Assume that \linebreak $(M,t) \in \epi(f(\|\cdot\|_{g,r\ast}))$. By \cref{lem:convhull},
\begin{align*}
\textstyle M = \sum_i \alpha_i M_i  \quad \text{with}\quad \sum_i \alpha_i = 1, \ \alpha_i \geq 0
\end{align*}
where $M_i$ satisfies 
\begin{equation*}
\rk(M_i) \leq r,  \quad \text{and} \quad \ \|M_i\|_{g,r\ast} = \|M\|_{g,r\ast}.
\end{equation*}
		Hence, $(M,t) = \sum_i \alpha_i \left(M_i,t\right)$, where
		\begin{equation*}
		 \ t \geq f(\|M\|_{g,r\ast}) = f(\|M_i\|_{g,r\ast}) \quad \text{and} \quad \rk(M_i) \leq r.
		\end{equation*}
		This shows that $(M_i,t)\in\epi(f_{\textnormal{reg}}),$ and therefore $(M,t)\in\conv(\epi(f_{\textnormal{reg}}))$.
		
		Conversely, if  $(M,t) \in \conv( \epi(f_{\textnormal{reg}})),$ then
		\begin{equation*}
	\textstyle	(M,t) = \sum_i \alpha_i \left(M_i,t_i \right) \quad \text{with} \quad \sum_{i} \alpha_i = 1, \alpha_i \geq 0,
		\end{equation*}
		where $M_i$ satisfies 
		\begin{equation*}
		\rk(M_i) \leq r, \quad \text{and} \quad t_i \geq f(\|M_i\|_{g}) = f(\|M_i\|_{g,r\ast}),
		\end{equation*}		
		 where the equality is due to \cref{eq:rank_norm} in \cref{lem:norm}. Since $f$ is convex and increasing, it holds that the composition $f(\|\cdot\|_{g,r\ast})$ is convex (see~\cite[Proposition~IV.2.1.8]{hiriart2013convex}). Thus,
		\begin{align*}
		t := \textstyle{\sum_i} \alpha_i t_i \geq \textstyle{\sum_i} \alpha_i f(\|M_i\|_{g,r\ast}) \geq   f\left( \normrgast{\textstyle{\sum_{i}}\alpha_iM_i}\right) = f\left(\|M\|_{g,r\ast} \right),
		\end{align*}
		which implies that $(M,t) \in \epi(f(\|\cdot\|_{g,r\ast})),$ and \cref{eq:biconj_reg} follows. Finally, \cref{eq:conj_reg} is proven by applying \cite[Theorem~15.3]{rockafellar1970convex} to $f(\|\cdot\|_{g,r\ast})$.
	\end{proof}
This result generalizes the corresponding result in \cite{grussler2016low}, in which the special case $f(x)\equiv x^2$ and $\normA{\cdot}{g}$ being the Frobenius norm is considered. For linear $f(x)\equiv x$, the biconjugate in \cref{eq:biconj_reg} reduces to the low-rank inducing norms of \Cref{sec:norms}. For the low-rank inducing Frobenius- and spectral norms, as well as for the nuclear norm, this yields the following characterizations.
\begin{cor}
Let $r\in \mathbb{N}$ be such that $1 \leq r \leq q:=\minmn$. Then
\begin{align*}
\| \cdot \|_{\ell_2,r\ast} &= (\|\cdot\|_F+\chi_{\rk(\cdot)\leq r})^{**},\\
\| \cdot \|_{\ell_\infty,r\ast} &= (\|\cdot\|_{\ell_\infty}+\chi_{\rk(\cdot)\leq r})^{**},
\end{align*}
and the nuclear norm satisfies
\begin{align*}
\|\cdot\|_{\ell_{1}} &= (\|\cdot\|_g+\chi_{\rk(\cdot)\leq 1})^{**},
\end{align*}
where $\|\cdot\|_g$ is an arbitrary unitarily invariant norm that satisfies $\|M\|_{g}=\sigma_1(M)$ for all rank-1 matrices $M$.
\end{cor}
\begin{proof}
Since $\|\cdot\|_{\ell_2}=\|\cdot\|_F$ is the Frobenius norm, this follows immediately from \cref{thm:main,prop:nuclear_norm_characterization}.
\end{proof}

\begin{rem}
This nuclear norm representation differs from the one in \cite{fazel2001rank,fazel2002matrix}, where it is shown that $\|\cdot\|_{\ell_1} + \chi_{B_{\ell_{\infty}}^1}=(\rk+\chi_{B_{\ell_{\infty}}^1})^{**}$, i.e. it is the convex envelope of the rank function restricted to the unit spectral norm ball.
\end{rem}
Using \cref{thm:main}, optimal convex relaxations of rank constrained problems
\begin{equation}
\begin{aligned}
& \underset{M}{\textnormal{minimize}}
& & f_0(M)+f(\|M\|_g)\\
& \textnormal{subject to}
& & \rk(M) \leq r,
\end{aligned}
\label{eq:low_rank_prob_main}
\end{equation}
can be provided, where $f_0: \funcdom$ is a proper and closed convex function and $f: \mathbb{R}_{\geq 0} \to \mathbb{R} \cup \{ \infty \}$ is an increasing and closed convex function. The problem in \cref{eq:low_rank_prob_main} is equivalent to minimizing $f_0+f_{\textnormal{reg}}$ with the non-convex $f_{\textnormal{reg}}$ defined in \cref{thm:main}. Therefore, the optimal convex relaxation of \cref{eq:low_rank_prob_main} is given by
\begin{align}
\underset{M}{\textnormal{minimize}}\quad f_0(M)+f(\|M\|_{g,r*}).
\label{eq:opt_conv_relax}
\end{align}
Including an additional regularization parameter $\theta \geq 0$ (that can be included in $f$) yields the following proposition.
	\begin{prop}
		\label{prop:opt_reg}
		Assume that $f_0: \funcdom$ is a proper closed convex function, and that $r \in \mathbb{N}$ is such that $1 \leq r \leq \minmn$. Let $f: \mathbb{R}_{\geq 0} \to \mathbb{R} \cup \{ \infty \}$ be an increasing, proper closed convex function, and let $\theta \geq 0$. Then
		\begin{align}
		\inf_{\stackrel{M \in \Rmn}{\rk(M) \leq r}} \left[ f_0(M) + \theta f(\|M\|_g) \right] &\geq \inf_{M \in \Rmn} \left[f_0(M) + \theta f(\|M\|_{g,r\ast}) \right] \label{eq:fenchel_inequ_main}.
		\end{align}
If $M^\opts$ solves the problem on the right such that $\rk(M^\opts) \leq r$, then equality holds, and $M^\opts$ is also a solution to the problem on the left. 
	\end{prop}
\begin{proof}
The inequality holds since $f(\|\cdot\|_{g,r*})=f_{\textnormal{reg}}^{**}\leq f_{\textnormal{reg}}$. From \cref{lem:norm} it follows that if $\rk(M^\opts)\leq r$ then $$f_{\textnormal{reg}}^{**}(M^\opts)=f(\|M^\opts\|_{g,r*})=f(\|M^\opts\|_g)=f_{\textnormal{reg}}(M^\opts),$$ which implies that the lower bound is attained with $M^\opts$ and the equality is valid.
\end{proof}

Since the nuclear norm is obtained by creating a low-rank inducing norm with $r=1$, we conclude that any nuclear norm regularized problem can be interpreted as an optimal convex relaxation to a non-convex problem of the form \cref{eq:low_rank_prob_main} with the constraint $\rk(M)\leq 1$. 

{{\cref{prop:opt_reg} can also be applied to the Frobenius norm approximation problem}
\begin{align*}
\inf_{\stackrel{M \in \Rmn}{\rk(M) \leq r}} \|N-M\|_F^2+h(M) =\inf_{\stackrel{M \in \Rmn}{\rk(M) \leq r}} \left[\|N\|_F^2 - 2\langle N,M \rangle + \|M\|_F^2 + h(M)\right],
\end{align*}
where $h$ is a proper convex function. Letting
\begin{equation*}
f_0(\cdot) = \|\cdot\|_F^2 - 2\langle N,\cdot \rangle+ h(\cdot), \quad f(x)=x^2, \quad \text{and} \quad \|\cdot\|_{g}=\|\cdot\|_F,
\end{equation*} yields

 	\begin{align*}
 	\min_{\stackrel{M \in \Rmn}{\rk(M) \leq r}} &\left[\|N-M\|_F^2 + h(M) \right] \geq \min_{{M \in \Rmn}} \left[\|N\|_F^2 - 2 \langle N,M \rangle+\|M\|_{\ell_2,r\ast}^2 + h(M) \right].
 	\end{align*} 
For $h = 0$, the solutions of the non-convex problem are determined by the celebrated Schmidt-Mirsky Theorem (see~\cite[Corollary~4.12]{stewart1990matrix}) as 
\begin{align*}
\svd_r(N) := \left\lbrace X = \sum_{i=1}^{r} \sigma_i(N) u_i v_i^\transp: N = \sum_{i=1}^{q} \sigma_i(N) u_i v_i^\transp \text{ is an SVD of } N\right\rbrace,
\end{align*}
where SVD stands for singular value decomposition. By \cref{lem:convhull} we have that \linebreak $\|X\|_{\ell_2,r\ast} = \|X\|_{F}$ for all $X \in \svd_r(N)$. This implies that the minima of the best convex relaxation are given by $\conv(\svd_r(N))$. Note that there is no analogue case for the spectral norm, because the spectral norm is not separable.}

%% file: sec_comp.txt
	\section{Computability}
	\label{sec:comp}

This section addresses the computability of convex optimization problems involving low-rank inducing regularizers of the form $f(\normrgast{\cdot})$. We restrict ourselves to low-rank inducing Frobenius- and spectral norm regularizers. A requirement for the optimal convex relaxation problem in \cref{eq:opt_conv_relax} to be solved efficiently, is that these regularizers are suitable for numerical optimization.

Assuming that $f_0$ and $f$ are SDP representable, it is shown that \cref{eq:opt_conv_relax} can be solved via semi-definite programming. To be able to solve larger problems using first-order proximal splitting methods (see~\cite{combettes2011proximal,parikh2014proximal}), it is shown in \cite{grussler2016lowrank} how to efficiently compute the proximal mappings of our regularizers. The computational cost of computing these proximal mappings is comparable to the cost of computing the proximal mapping for the nuclear norm, since the cost in all cases is dominated by the singular value decomposition. 

In order to deal with increasing convex functions $f$ in \cref{eq:opt_conv_relax}, the problem is rewritten into the equivalent epigraph form
\begin{equation}\label{eq:epi_form}
\underset{M,v}{\textnormal{minimize~}} f_0(M)+f(v)+\chi_{\epi(\normrgast{\cdot})}(M,v).
\end{equation}

\subsection{SDP representation}
The low-rank inducing Frobenius norm and spectral norm
\begin{align}
	\| M \|_{\ell_2,r\ast} &:= \max_{\|Y \|_{\ell_2,r} \leq 1}  \langle M,Y \rangle =   \max_{\|Y \|_{\ell_2,r}^2 \leq 1}  \langle M,Y \rangle, \label{eq:rast_SDP}\\
	\normkyrast{M} &:= \max_{\normkyr{Y} \leq 1}  \langle M,Y \rangle\label{eq:l1rast_SDP}
	\end{align}	
are SDP representable via $\|Y\|_{\ell_2,r}^2$ and $\normkyr{Y}$. From~\cite{grussler2015optimal,grussler2016low}, it is known that
	\begin{equation*}
	\begin{aligned}
	\|Y\|_{\ell_2,r}^2=~& \underset{T,\gamma}{\min}
	& & \trace(T) - \gamma (n-r)\\
	& \textnormal{s.t.}
	& & \begin{pmatrix}
	T &Y \\
	Y^{\transp} & I
	\end{pmatrix} \succeq 0, \ T \succeq \gamma I.
	\end{aligned}
	\end{equation*}
	Similarly, one can verify that
	\begin{equation*}
	\begin{aligned}
	\normkyr{Y}=~& \underset{T_1,T_2,\gamma}{\min }
	& & \frac{1}{2}[\trace(T_1) + \trace(T_2) - (n+m-2r)\gamma ]\\
	& \textnormal{s.t.}
	& & \begin{pmatrix}
	T_1 &Y \\
	Y^{\transp} & T_2
	\end{pmatrix} \succeq 0, \ T_1, T_2 \succeq \gamma I,
	\end{aligned}
	\end{equation*}
	which generalizes the SDP representation of $\|Y\|_{\ell_1}$  in~\cite{recht2010guaranteed}. This implies that
		\begin{align*}
		&\begin{aligned}
		\|M\|_{\ell_2,r\ast}=~& \underset{Y,T,\gamma}{\max}
		& & \langle M, Y\rangle\\
		& \textnormal{s.t.}
		& & \begin{pmatrix}
		T &Y \\
		Y^{\transp} & I
		\end{pmatrix} \succeq 0, \ T \succeq \gamma I,\\
		& & & \trace(T) - \gamma (n-r) \leq 1,
		\end{aligned}\\[4mm]
		&\begin{aligned}
		\|M\|_{\ell_{\infty},r*}=~& \underset{Y,T_1,T_2,\gamma}{\max}
		& & \langle M, Y\rangle\\
		& \textnormal{s.t.}
		& & \begin{pmatrix}
		T_1 &Y \\
		Y^{\transp} & T_2
		\end{pmatrix} \succeq 0, \ T_1,T_2 \succeq \gamma I,\\
		& & & \frac{1}{2}[\trace(T_1) + \trace(T_2) - (n+m-2r)\gamma ]\leq 1.
		\end{aligned}
		\end{align*}
However, these formulations cannot be used in convex optimization problems with $M$ as a decision variable due to the inner product $\langle M, Y\rangle$. Therefore, we use duality to arrive at
		\begin{align*}
		&\begin{aligned}
		\|M\|_{\ell_2,r\ast}=~& \underset{W_1,W_2,k}{\min }
		& & \frac{1}{2}(\trace(W_2)+k)\\
		& \textnormal{s.t.}
		& & \begin{pmatrix}
		kI-W_1 & M \\
		M^{\transp} & W_2
		\end{pmatrix} \succeq 0, \ W_1 \succeq 0, \\
		& & & \trace(W_1) = (n-r)k;
		\end{aligned}\\[4mm]
		&\begin{aligned}
		\|M\|_{\ell_{\infty},r*}=~& \underset{W_1,W_2,k}{\min }
		& & k\\
		& \textnormal{s.t.}
		& & \begin{pmatrix}
		kI-W_1 & M \\
		M^{\transp} & kI-W_2
		\end{pmatrix} \succeq 0, \ W_1,W_2 \succeq 0, \\
		& & & \trace(W_1) + \trace(W_2) = [(n-r)+(m-r)]k.
		\end{aligned}
		\end{align*}
These formulations can be used to, e.g. solve problems on the epigraph form \cref{eq:epi_form} by enforcing the respective costs to be smaller than or equal to $v\in\mathbb{R}$. This gives constraints of the form $\|M\|_{g,r*}\leq v$, i.e. $(M,v)\in\epi(\|\cdot\|_{g,r*})$. If $f$ and $f_0$ are SDP representable, then \cref{eq:epi_form} can be solved via semi-definite programming.
	\subsection{Splitting algorithms}

		Conventional SDP solvers are often based on interior point methods (see~\cite{peaucelle2002user,toh2004implementation}). These have good convergence properties, but the iteration complexity typically grows unfavorably with the problem dimension. This limits their application to small or medium scale problems. First order proximal splitting methods (see e.g.~\cite{combettes2011proximal,parikh2014proximal}) typically have a lower complexity per iteration, and are thus more suitable for large problems.
		
These methods require the proximal mapping for all non-smooth parts of the problem to be available and cheaply computable. For closed, proper and convex functions $h: \funcdom$, the \emph{proximal mapping} is defined as
		\begin{align}
		\prox_{\gamma h}(Z) := \argmin_X \left(h(X)+\dfrac{1}{2\gamma} \|X-Z\|^2_F \right). \label{eq:prox_op}
		\end{align}
Applying proximal splitting methods to~\cref{eq:epi_form} therefore requires that the proximal mapping of $\chi_{\epi(\normrgast{\cdot})}$ is readily computable.
Algorithms for efficiently computing these proximal mappings in case of the low-rank inducing Frobenius- and spectral norms are derived in~\cite{grussler2016lowrank} (see~\cite{grussler2018github} for implementations). 
%
%
%
%
%
Finally, the detour over the epigraph projection is not needed for all increasing functions. The proximal mapping for the low-rank inducing Frobenius- and spectral norms can be derived very similarly to the epigraph case. 

%% file: sec_ex_v2.txt
\section{Examples: Matrix Completion}
\label{sec:compl}
The matrix completion problem seeks to complete a low-rank matrix based on limited knowledge about its entries. The problem is often posed as
	\begin{equation}
	\begin{aligned}
	& {\textnormal{minimize}}
	& & \rk(X)\\
	& \textnormal{subject to}
	& & \hat{x}_{ij} = x_{ij}, \ (i,j) \in \mathcal{I}, \label{prob:mat_compl}	\end{aligned}
	\end{equation}
where $\mathcal{I}$ denotes the index set of the known entries. Another formulation that fits with the low-rank inducing norms proposed in this paper is 
\begin{equation}
	\begin{aligned}
	& {\textnormal{minimize}}
	& & \|X\|_g\\
	& \textnormal{subject to}
        & & \rk(X)\leq r,\\
	& & & \hat{x}_{ij} = x_{ij}, \ (i,j) \in \mathcal{I}, \label{prob:mat_compl_rank_constr}	\end{aligned}
	\end{equation}
where $r$ is the target rank of the matrix to be completed. In the following, two examples of this form will be convexified through low-rank inducing Frobenius- and spectral norms. That is,
\begin{equation}
	\begin{aligned}
	& {\textnormal{minimize}}
	& & \normrgast{X}\\
	& \textnormal{subject to}
	& & \hat{x}_{ij} = x_{ij}, \ (i,j) \in \mathcal{I}, \label{prob:mat_compl_convex}	\end{aligned}
	\end{equation}
is solved for $\normrgast{\cdot} = \normAast{\cdot}{\ell_2,r}$ and $\normrgast{\cdot} = \normAast{\cdot}{\ell_\infty,r}$.

Further, we discuss a covariance completion problem which is a generalization of the problem above. In all problems it will be observed that there are convex relaxations with low-rank inducing norms whose solutions give better completion than the nuclear norm approach.

\subsection{Example 1}
In the first problem, the matrix $\hat{X}$ to be completed is a low-rank approximation of the Hankel matrix
\begin{equation}
H=
\begin{tikzpicture}[baseline=(current bounding box.center)]
\matrix (m) [matrix of math nodes,
			 nodes in empty cells,
			 right delimiter={)},
			 left delimiter={(}]{
	1  	& 	1 	& 	  	&   	& 	1 	& 	1  \\
	1  	& 		& 		& 		&  	    & 	0	\\
		&		&		&		&		&		\\
		&		&		&		&		& 		\\
	1	&		&		&		&		& 	0  	\\
	1   &  	0   &		&		&	0	&	0	\\
} ;
\newdimen\L
\L = .8 pt
\draw[loosely dotted, line width = \L] (m-1-2)-- (m-1-5);
\draw[loosely dotted, line width = \L] (m-6-2)-- (m-6-5);
%
\draw[loosely dotted, line width = \L] (m-2-1)-- (m-5-1);
\draw[loosely dotted, line width = \L] (m-2-6)-- (m-5-6);
%
\draw[loosely dotted, line width = \L] (m-5-1)-- (m-1-5);
\draw[loosely dotted, line width = \L] (m-6-1)-- (m-1-6);
\draw[loosely dotted, line width = \L] (m-6-2)-- (m-2-6);


\end{tikzpicture} \in\mathbb{R}^{10\times 10}.
\label{eq:Hankel}
\end{equation}
Let the singular value decomposition of $H$ be given by $H = \sum_{i=1}^{10} \sigma_i(H) u_i u_i^{\transp}$ and
\begin{equation*}
\hat{X} := \sum_{i=1}^{5} \sigma_i(H) u_i u_i^{\transp} \quad \text{and } \quad \mathcal{I}:=\{(i,j):\hat{x}_{ij}>0\},
\end{equation*}
where $\mathcal{I}$ is the index set of known entries. The cardinality of $\mathcal{I}$ is 78, i.e. 22 out of 100 entries are unknown. The choice of $\hat{X}$ can be understood as follows. {{By the definition of the spectral norm (as the induced norm of the Euclidean norm) it holds that}
\begin{align*}
|\hat{x}_{ij} - h_{ij}| \leq \sigma_{r+1}(H),
\end{align*}
which implies that $|\hat{x}_{ij}| \leq \sigma_{r+1}(H) \text{ for all } (i,j) \in \mathcal{I}$. Thus the majority of the unknown entries can be expected to be of significantly smaller magnitude than the magnitudes of the known elements. Indeed, $64$ out of $78$ known entries have larger magnitude than any of the unknown elements. With such prior knowledge at hand, it seems natural to model \cref{prob:mat_compl_rank_constr} with $\| \cdot \| = \| \cdot\|_F$.
	}

\begin{figure}[t]
	\def\factor{.4}%
	\begin{tikzpicture}
	\begin{groupplot}[group style={group name=my plots, group size=2 by 1,horizontal sep=2.5 cm},width=\factor \textwidth]
	\nextgroupplot[xlabel=$r$,
	ylabel=$\dfrac{\|X-\hat{X}\|_F}{\|X\|_F}$,
	ymin=0,
	grid = both,
	xmin = 1,
	xmax = 10,
	xtick = {1,...,10},
	ytick = {1e-12,1e-8,1e-4,1e0},
	scaled ticks=false, 
	ymode = log,
	tick label style={/pgf/number format/fixed}]
	\addplot+[color = FigColor1 ,mark options = {solid}] file{compl_rastbest_err_rast.txt};
	\label{line:compl_rastbest_err_rast}
	\addplot+[color = FigColor2 ,mark options = {solid}, dashed] file{compl_rastbest_err_spec.txt};
	\label{line:compl_rastbest_err_spec}
	\coordinate (nl) at (current axis.north);
	\nextgroupplot[xlabel=$r$,
	ylabel=$\rk(X)$,
	ymin=5,
	ymax=10,
	ytick = {5,6,7,8,9,10},
	grid = both,
	xmin = 1,
	xmax = 10,
	xtick = {1,...,10},
	scaled ticks=false, 
	tick label style={/pgf/number format/fixed}]
	\addplot+[color = FigColor1 ,mark options = {solid}] file{compl_rastbest_rank_rast.txt};
	\label{line:compl_rastbest_rank_rast}
	\addplot+[color = FigColor2 ,mark options = {solid},dashed] file{compl_rastbest_rank_spec.txt};
	\label{line:compl_rastbest_rank_spec}
	\coordinate (nl) at (current axis.north);
	\end{groupplot}
	
	\node[text width=5cm ,align=left,anchor=north] at ([yshift=-7mm]my plots c1r1.south) {\subcaption{Relative completion errors of: \newline
			\ref{line:compl_rastbest_err_rast} \cref{prob:mat_compl_convex} with $\normrgast{\cdot}=\| \cdot \|_{\ell_2,r\ast}$\newline
			\ref{line:compl_rastbest_err_spec} \cref{prob:mat_compl_convex} with $\normrgast{\cdot}=\normkyrast{\cdot}$
			\label{fig:err_rastbest}}};%
	\node[text width=5cm ,align=left,anchor=north] at ([yshift=-7mm]my plots c2r1.south) 
	{\subcaption{Rank of the solutions to: \newline
			\ref{line:compl_rastbest_rank_rast} \cref{prob:mat_compl_convex} with $\normrgast{\cdot}=\| \cdot \|_{\ell_2,r\ast}$\newline
			\ref{line:compl_rastbest_rank_spec} \cref{prob:mat_compl_convex} with $\normrgast{\cdot}=\normkyrast{\cdot}$
			\label{fig:rank_rastbest}}};%
	\end{tikzpicture}
	\caption{Example 1: Relative completion error and ranks of the solution to \cref{prob:mat_compl_convex} with $\normrgast{\cdot}=\| \cdot \|_{\ell_2,r\ast}$ and $\normrgast{\cdot}=\normkyrast{\cdot}$. \label{fig:rastbest}}
\end{figure}

\Cref{fig:rastbest} shows the completion errors and ranks of the completed matrices for different values of $r$. The nuclear norm ($r=1$) returns a full rank matrix and gives a worse completion error than all other low-rank inducing Frobenius norms. For $r=5$, the solution with the low-rank inducing Frobenius norm has rank 5. Given the known entries, this is the matrix of smallest Frobenius norm which has at most rank 5, by \cref{prop:opt_reg}. As indicated by the small relative error, this matrix coincides with $\hat{X}$. 

Notice that 
$$226 = 3 r (20 - r) + 1 \gg \card(\mathcal{I}) = 78,$$
which is why exact completion results for the nuclear norm (see~\cite[Proposition
3.11]{chandrasekaran2012convex}) do not apply. Finally, note that the low-rank inducing spectral norm shows no improvement in comparison with the nuclear norm.

\subsection{Example 2}
In this second example, we let
\begin{equation*}
\hat{X} := \sum_{j=1}^{5} \sigma_j(H) \sum_{i=1}^{5} u_i v_i^{\transp} \quad \text{and} \quad \mathcal{I}:=\{(i,j):\hat{x}_{ij}>0\},
\end{equation*}
where $H$ is given in~\cref{eq:Hankel} with the singular value decomposition $H = \sum_{i=1}^{10} \sigma_i(H) u_i v_i^{\transp}$. The cardinality of $\mathcal{I}$ is 67, that is, 33 out of 100 entries are unknown. {{By construction of $\hat{X}$ it follows by \cref{prp:rank_r_ind_spectral,prp:rank_r_ind_str_conv} that $\hat{X}$ is an extreme point of the balls
\begin{align*}
\lbrace X: \|X\|_{\ell_\infty,5\ast} \leq \|\hat{X}\|_{\ell_\infty}  \rbrace \quad \text{and} \quad \lbrace X: \|X\|_{\ell_2,5\ast} \leq \|\hat{X}\|_{F} \rbrace.
\end{align*}
In contrast to Example 1, the unknown entries are this time of comparable magnitude as the known ones, which is why modeling \cref{prob:mat_compl_rank_constr} with $\| \cdot \|_{g} = \| \cdot \|_{F}$ does not seem to be a desired choice. Instead, we will see that the choice $\| \cdot \|_{g} = \| \cdot \|_{\ell_\infty}$ is more effective.}}
\begin{figure}[t]
	\def\factor{.4}%
	\begin{tikzpicture}
	\begin{groupplot}[group style={group name=my plots, group size=2 by 1,horizontal sep=2.5 cm},width=\factor \textwidth]
	\nextgroupplot[xlabel=$r$,
	ylabel=$\dfrac{\|X-\hat{X}\|_F}{\|X\|_F}$,
	ymin=0,
	grid = both,
	xmin = 1,
	xmax = 10,
	xtick = {1,...,10},
	scaled ticks=false, 
	tick label style={/pgf/number format/fixed}]
	\addplot+[color = FigColor1 ,mark options = {solid}] file{compl_specbest_err_rast.txt};
	\label{line:compl_specbest_err_rast}
	\addplot+[color = FigColor2 ,mark options = {solid}, dashed] file{compl_specbest_err_spec.txt};
	\label{line:compl_specbest_err_spec}
	\coordinate (nl) at (current axis.north);
	\nextgroupplot[xlabel=$r$,
	ylabel=$\rk(X)$,
	ymin=5,
	ymax=10,
	ytick = {5,6,7,8,9,10},
	grid = both,
	xmin = 1,
	xmax = 10,
	xtick = {1,...,10},
	scaled ticks=false, 
	tick label style={/pgf/number format/fixed}]
	\addplot+[color = FigColor1 ,mark options = {solid}] file{compl_specbest_rank_rast.txt};
	\label{line:compl_specbest_rank_rast}
	\addplot+[color = FigColor2 ,mark options = {solid},dashed] file{compl_specbest_rank_spec.txt};
	\label{line:compl_specbest_rank_spec}
	\coordinate (nl) at (current axis.north);
	\end{groupplot}
	\node[text width=5cm ,align=left,anchor=north] at ([yshift=-7mm]my plots c1r1.south) {\subcaption{Relative completion errors of: \newline
			\ref{line:compl_specbest_err_rast} \cref{prob:mat_compl_convex} with $\|\cdot\|_{g,r*}=\|\cdot\|_{\ell_2,r\ast}$\newline
			\ref{line:compl_specbest_err_spec} \cref{prob:mat_compl_convex} with $\|\cdot\|_{g,r*}=\normkyrast{\cdot}$
			\label{fig:err_specbest}}};%
	\node[text width=5cm ,align=left,anchor=north] at ([yshift=-7mm]my plots c2r1.south) 
	{\subcaption{Rank of the solutions to: \newline
			\ref{line:compl_specbest_rank_rast} \cref{prob:mat_compl_convex} with $\|\cdot\|_{g,r*}=\|\cdot\|_{\ell_2,r\ast}$\newline
			\ref{line:compl_specbest_rank_spec} \cref{prob:mat_compl_convex} with $\|\cdot\|_{g,r*}=\normkyrast{\cdot}$
			\label{fig:rank_specbest}}};%
	\end{tikzpicture}
	\caption{Example 2: Relative completion error and ranks of the solution to \cref{prob:mat_compl_convex} with \newline $\|\cdot\|_{g,r*}=\|\cdot\|_{\ell_2,r\ast}$ and $\|\cdot\|_{g,r*}=\normkyrast{\cdot}$. \label{fig:specbest}}
\end{figure}
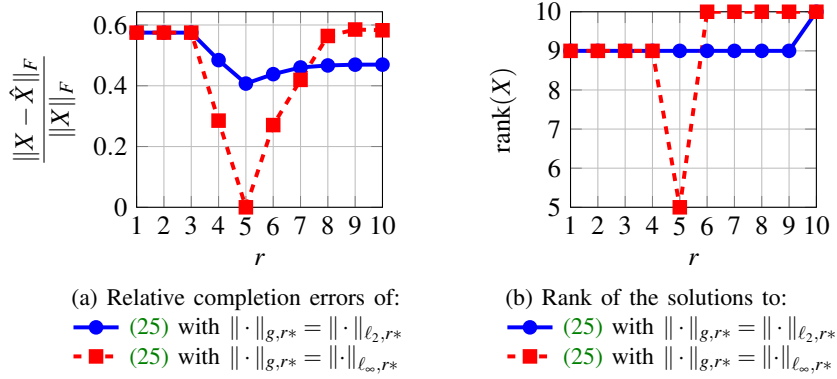
\Cref{fig:specbest} shows the completion errors and ranks of the completed matrices with different value of $r$. The nuclear norm ($r=1$) returns a close to full rank matrix with a relative completion error that is among the largest for all $r$. In this example, the low-rank inducing spectral norms perform significantly better than the low-rank inducing Frobenius norms. In particular, for $r=5$, the low-rank inducing spectral norm returns a rank 5 solution. As in the previous examples, this solution is the matrix of smallest spectral norm of rank at most 5, given the known entries (see~\cref{prop:opt_reg}). The zero completion error shows that this matrix coincides with $\hat{X}$. 
 Analogous to the previous example, $$226 = 3 r (20 - r) + 1  \gg \card(\mathcal{I}) = 67,$$
which is why exact completion with the nuclear norm cannot be expected. 

In both examples, the nuclear norm neither produces the lowest rank solution, nor recovers the true matrix. In contrast, other low-rank inducing norms succeed in both aspects. This indicates that the richness in the family of low-rank inducing norms should be exploited to achieve satisfactory performance in rank constrained problems. In practical applications, the 'true' matrix is not known, and this comparison cannot be made. However, cross validation techniques may be used to assess the performance. Finally note that both examples can be scaled to higher dimensions while leaving our conclusions invariant.

\subsection{Covariance Completion}
\label{sec:ex}

In this section, the performance of the low-rank inducing Frobenius- and spectral norms is evaluated by means of a covariance completion problem, which is taken from \cite{zare2014low}. This is a variation of the matrix completion problems above and {{shall illustrate two important features of our low-rank inducing norms for non-ideal situations. Firstly, these norms still give reasonable results; secondly, the choice of norm is crucial, even if other norms return zero duality gaps.}}

Consider the linear state-space system
\begin{align*}
\dot{x}(t) = Ax(t) + Bu(t),
\end{align*}
with $A \in \Rnn$, $B \in \Rmn$, $m\leq n$ and $u(t)$ is a zero-mean stationary stochastic process. For Hurwitz $A$ and reachable $(A,B)$, it has been shown (see \cite{georgiou2002spectral,georgiou2002structure}) that the following are equivalent:
\begin{enumerate}[i.]
	\item $X  := \lim_{t\,\rightarrow \,\infty} \mathbf{E} \left( x(t)\, x^{\transp}(t) \right) \succeq 0$ is the steady-state covariance matrix of $x(t)$, where $\mathbf{E} (\cdot)$ denotes the expected value.
	\item $\exists H \in \mathbb{R}^{m \times n}: AX + XA^{\transp} = -(BH + H^{\transp}B^{\transp})$.
	\item $\rk\begin{pmatrix}
	AX + XA^{\transp} & B\\
	B^{\transp} & 0
	\end{pmatrix} = \rk \begin{pmatrix}
	0 & B\\
	B^{\transp} & 0
	\end{pmatrix}$. 
\end{enumerate}
In particular, $H = \frac{1}{2} \mathbf{E} \left( u(t)\, u^{\transp}(t) \right)  B^{\transp}$ if $u$ is white noise. The problem considered in~\cite{chen2013state,lin2013admm,zare2014low,zare2015alternating,zare2016color} is to reconstruct the partially known covariance matrix $X$ and the input matrix $B$, via $M=-(BH+H^{\transp}B^{\transp})$, where the rank of $M$ sets an upper bound on the rank of $B$, i.e. the number of inputs. The objective is to keep the rank of $M$ low, while achieving satisfactory completion of $X$. In \cite{chen2013state,lin2013admm,zare2014low,zare2015alternating,zare2016color} the problem is addressed by searching for the lowest rank solution:
	\begin{equation}
	\begin{aligned}
	& {\textnormal{minimize}}
	& & \rk(M)\\
	& \textnormal{subject to}
	& & \hat{x}_{ij} = x_{ij}, \ (i,j) \in \mathcal{I}, \label{prob:compl_cov}\\
	& & & A\hat{X} + \hat{X}A^{\transp} = -M,\\
	& & & \hat{X} \succeq 0,
	\end{aligned}
	\end{equation}
where $\mathcal{I}$ denotes set of pairs of indices of known entries. Another option is to search for a low-rank solution, while minimizing the norm of $M$ measured by some unitarily invariant norm. This helps to avoid overfitting, and gives
	\begin{equation}
	\begin{aligned}
	& {\textnormal{minimize}}
	& & \|M\|_g\\
	& \textnormal{subject to} 
        & & \rk(M)\leq r,\\
	& & & \hat{x}_{ij} = x_{ij}, \ (i,j) \in \mathcal{I}, \label{prob:compl_cov_g}\\
	& & & A\hat{X} + \hat{X}A^{\transp} = -M,\\
	& & & \hat{X} \succeq 0.
	\end{aligned}
	\end{equation}
The authors in~\cite{chen2013state,lin2013admm,zare2014low,zare2015alternating,zare2016color} convexify the problem by using the nuclear norm. 
In~\cite{grussler2016covariance}, a similar problem is instead convexified with the low-rank inducing Frobenius norm. We will make a comparison with convex relaxations based on low-rank inducing spectral norms. All these convex relaxations are of the form
\begin{equation}
\begin{aligned}
& {\textnormal{minimize}}
& & \|M\|_{g,r*}\\
& \textnormal{subject to}
& & \hat{x}_{ij} = x_{ij}, \ (i,j) \in \mathcal{I}, \label{opt:compl_cov_r}\\
& & & A\hat{X} + \hat{X}A^{\transp} = -M,\\
& & & \hat{X} \succeq 0,
\end{aligned}
\end{equation}
with the appropriate low-rank inducing norm in the cost.

\subsubsection{Mass-spring-damper system}
The system considered in our example is the so-called \emph{mass-spring-damper system (MSD)} (see~\cite{zare2015alternating,grussler2016covariance}) with $n$ masses (see~\Cref{fig:MSD}). 


\begin{figure}
	\centering
\begin{tikzpicture}[every node/.style={draw,outer sep=0pt,thick},force/.style={>=latex,draw=black,fill=black,thick}]
\tikzstyle{spring}=[thick,decorate,decoration={zigzag,pre length=0.3cm,post length=0.3cm,segment length=4}]
\tikzstyle{damper}=[thick,decoration={markings,  
	mark connection node=dmp,
	mark=at position 0.5 with 
	{
		\node (dmp) [thick,inner sep=0pt,transform shape,rotate=-90,minimum width=8pt,minimum height=3pt,draw=none] {};
		\draw [thick] ($(dmp.north east)+(3pt,0)$) -- (dmp.south east) -- (dmp.south west) -- ($(dmp.north west)+(3pt,0)$);
		\draw [thick] ($(dmp.north)+(0,-2.5pt)$) -- ($(dmp.north)+(0,2.5pt)$);
	}
}, decorate]
\tikzstyle{ground}=[fill,pattern=north east lines,draw=none,minimum width=0.75cm,minimum height=0.3cm]

\node (wall) [ground, rotate=-90, minimum width=2.1cm,yshift=-2cm] {};
\draw (wall.north east) -- (wall.north west);

\node (M) [minimum width=1cm, minimum height=1cm] {$m$};

\node (ground) [ground,anchor=north,yshift=-0.25cm,minimum width=2cm] at (M.south) {};
\draw (ground.north east) -- (ground.north west);

\draw [thick] (M.south west) ++ (0.2cm,-0.125cm) circle (0.125cm)  (M.south east) ++ (-0.2cm,-0.125cm) circle (0.125cm);

\draw [spring] (wall.150) -- ($(M.north west)!(wall.150)!(M.south west)$);
\draw [damper] (wall.25) -- ($(M.north west)!(wall.25)!(M.south west)$);

\draw[force,->] (M.north west) ++ (0,7pt) -- ++(1,0);
\node[above,draw = none, yshift = 21 pt] {$u_1$};

\node (M2) [minimum width=1cm, minimum height=1cm, xshift = 2.35 cm] {$m$};

\node (ground) [ground,anchor=north,yshift=-0.25cm,minimum width=2.75cm] at (M2.south) {};
\draw (ground.north east) -- (ground.north west);

\draw [thick] (M2.south west) ++ (0.2cm,-0.125cm) circle (0.125cm)  (M2.south east) ++ (-0.2cm,-0.125cm) circle (0.125cm);

\draw [spring] ($(M.north east)!(wall.150)!(M.south east)$)--($(M2.north west)!(wall.150)!(M2.south west)$);
\draw [damper] ($(M.north east)!(wall.25)!(M.south east)$)--($(M2.north west)!(wall.25)!(M2.south west)$);

\draw[force,->] (M2.north west) ++ (0,7pt) -- ++(1,0); 
\node[above,draw = none, yshift = 21 pt, xshift = 2.35cm]{$u_2$};

\node (M3) [draw = none,minimum width=1cm, minimum height=1cm, xshift = 4.7 cm] {$\dots$};
\node (ground) [ground,anchor=north,yshift=-0.25cm,minimum width=2.75cm] at (M3.south) {};
\draw (ground.north east) -- (ground.north west);

\draw [spring] ($(M2.north east)!(wall.150)!(M2.south east)$)--($(M3.north west)!(wall.150)!(M3.south west)$);
\draw [damper] ($(M2.north east)!(wall.25)!(M2.south east)$)--($(M3.north west)!(wall.25)!(M3.south west)$);


\node (M4) [minimum width=1cm, minimum height=1cm, xshift = 7.05 cm] {$m$};
\node (ground) [ground,anchor=north,yshift=-0.25cm,minimum width=2cm] at (M4.south) {};
\draw (ground.north east) -- (ground.north west);

\draw [spring] ($(M3.north east)!(wall.150)!(M3.south east)$)--($(M4.north west)!(wall.150)!(M4.south west)$);
\draw [damper] ($(M3.north east)!(wall.25)!(M3.south east)$)--($(M4.north west)!(wall.25)!(M4.south west)$);

\draw [thick] (M4.south west) ++ (0.2cm,-0.125cm) circle (0.125cm)  (M4.south east) ++ (-0.2cm,-0.125cm) circle (0.125cm);

\draw[force,->] (M4.north west) ++ (0,7pt) -- ++(1,0); 
\node[above,draw = none, yshift = 21 pt, xshift = 7cm]{$u_n$};

%
%

\node (wall2) [ground, rotate=-270, minimum width=2.1cm,yshift=-9.05 cm] {};
\draw (wall2.north east) -- (wall2.north west);
\draw [spring] ($(M4.north east)!(wall.150)!(M4.south east)$)--(wall2.30);
\draw [damper] ($(M4.north east)!(wall.25)!(M4.south east)$)--(wall2.155);

\end{tikzpicture}
\caption{Mass-spring-damper system with $n$ masses and input forces $u_1,\dots,u_n$.}
\label{fig:MSD}
\end{figure}
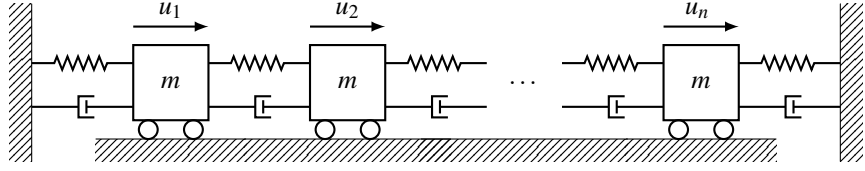

Assuming that the stochastic forcing affects all masses, this yields the following state-space representation
\begin{align*}
\dot{x}(t) &= Ax(t) + B \xi(t)
\end{align*}
with
\begin{align*}
A = \begin{pmatrix}
0 & I\\
-S & -I
\end{pmatrix} \in \mathbb{R}^{2n \times 2n}, \quad B = \begin{pmatrix}
0\\
I
\end{pmatrix} \in \mathbb{R}^{2n \times n}.
\end{align*}
Here, $S$ is a symmetric tridiagonal Toeplitz matrix with $2$ on the main diagonal, $-1$ on the first upper and lower sub-diagonals, and $I$ and $0$ stand for the identity and zero matrices of appropriate size. The state vector $x$ consists of the positions and velocities of the masses, $x = (p,v)$. 
Furthermore, $\xi(t)$ is generated via a low-pass filtered white noise signal $w(t)$ with unit covariance $\mathbf{E} \left( w(t) w(t)^{\transp} \right) = I$ as
\begin{align*}
\dot{\xi}(t) = -\xi(t) + w(t).
\end{align*}
The extended covariance matrix
\begin{align*}
X_e := \mathbf{E} \left( x_e x_e^{\transp} \right) = \begin{pmatrix}
X & X_{x \xi} \\
X_{\xi x} & X_{\xi}
\end{pmatrix} \ \textnormal{ with } \ x_e := \begin{pmatrix}
x(t)\\
\xi(t)
\end{pmatrix}
\end{align*}
is then determined by 
\begin{align*}
A_eX_e + X_eA_e^{\transp} = -B_e B_e^{\transp},
\end{align*} 
where $X$ is the steady-state covariance matrix of $x(t)$ and
\begin{align*}
A_e := \begin{pmatrix}
A & B \\
0 & -I
\end{pmatrix}, \quad B_e := \begin{pmatrix}
0\\
I
\end{pmatrix}.
\end{align*}


In our numerical experiments, we choose $n=20$ masses and assume that only one-point correlations are available, i.e. the known entries are given by the diagonal of $X$. The steady-state covariance matrix can be partitioned as
\begin{align*}
X = \begin{pmatrix}
X_{pp} & X_{pv}\\
X_{vp} & X_{vv}
\end{pmatrix},
\end{align*}
where $X_{pp}$ and $X_{vv}$ are the covariance matrices of the positions and the velocities, respectively. To visualize the effects of using different low-rank inducing norms in~\cref{opt:compl_cov_r}, an interpolated colormap of the reconstructed $\hat{X}_{pp}$ and $\hat{X}_{vv}$ is used (see~\Cref{fig.MSDcompletion}). \begin{figure}[t]
	\pgfplotsset{
		colorbar shift/.style={xshift=0cm}
	}
	\begin{center}
		\def\factor{.25}%
		\begin{tikzpicture}
		\begin{groupplot}[xmin = 1,xmax =20,
		ymin = 1,ymax =20,
		ytick = {16,11,6,1},
		yticklabels = {5,10,15,20},
		colormap/jet,
		every colorbar/.append style={
			yticklabel style={
				text width=1.75 em,
				align=left,
				/pgf/number format/.cd,
				fixed,
				fixed zerofill,
				precision=2}},
		group style={
			group name=left plots,
			group size=2 by 1,
			horizontal sep= 2.5cm,
			group name=my plots},
		width=\factor \textwidth,
		height=\factor \textwidth,
		scale only axis,
		subtitle/.style={title=\gpsubtitle{#1}},
		title style={at={([yshift=-5ex]0.5,0)},anchor=north}
		]
		\nextgroupplot[axis equal image,colorbar,colorbar style ={ytick ={.5,1,1.5,2},yticklabel style={
				/pgf/number format/.cd,
				fixed,
				fixed zerofill,
				precision=1}},point meta min=0.0238,
		point meta max=2.4747]
		\addplot[thick,blue] graphics[xmin=1,ymin=1,xmax=20,ymax=20] {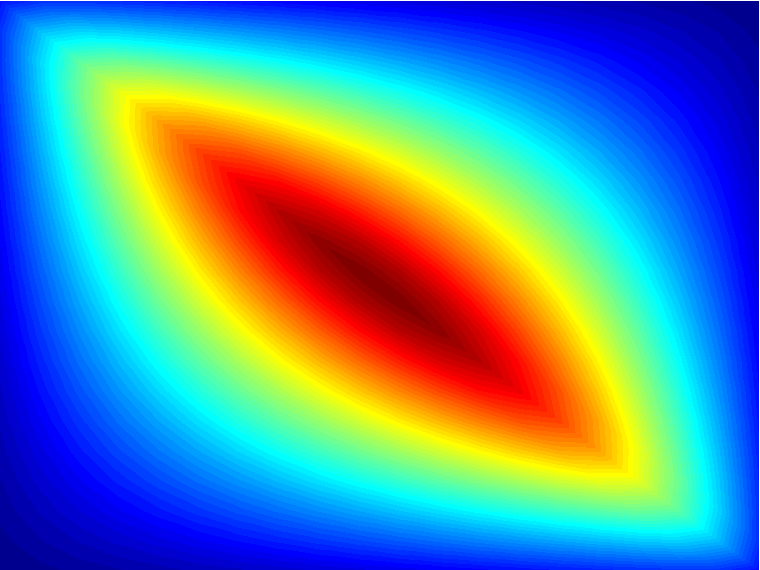};
		\addplot[line width = 1.5pt, black, domain = 0:21] (x,21-x); \label{cov_compl_MSD_black}
		\coordinate (nl) at (current axis.north);
		\nextgroupplot[axis equal image,colorbar,colorbar style ={ytick ={-.1,-.05,0,.05,.1}},point meta min=-0.1260,
		point meta max= 0.1443]
		\addplot[thick,blue] graphics[xmin=1,ymin=1,xmax=20,ymax=20] {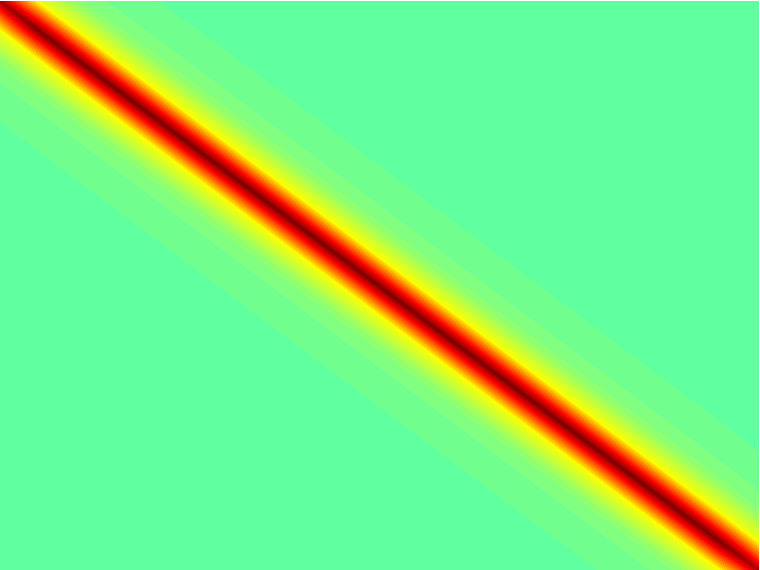};
		\addplot[line width = 1.5pt, black, domain = 0:21] (x,21-x); \coordinate (nl) at (current axis.north);
		\end{groupplot}
		\node[text width=5cm ,align=center,anchor=north] at ([yshift=-2mm]my plots c1r1.south) {\subcaption{$X_{pp}$ \label{fig.Sigmapp}}};%
		\node[text width=5cm,align=center,anchor=north] at ([yshift=-2mm]my plots c2r1.south) {\subcaption{$X_{vv}$ \label{fig.Sigmavv}}};%
		\end{tikzpicture}%
	\end{center}
	\caption{Interpolated colormap of the steady-state covariance matrices $X_{pp}$ and $X_{vv}$ of the positions and the velocities in the MSD system with $n=20$. \ref{cov_compl_MSD_black} indicates the available one-point correlations.}\label{fig.Sigma_pp_vv}
\end{figure}
The interpolated colormap of the true covariance matrices is shown in~\Cref{fig.Sigma_pp_vv}, where the black lines indicate the known measured entries. 
\begin{figure}[tbhp]
	
	\def\factor{.39}%
	\begin{tikzpicture}
	\begin{groupplot}[group style={group name=my plots, group size=2 by 1,horizontal sep=2.5 cm},ymin=0.0,no markers, ymax = 0.6,width=\factor \textwidth]
	\nextgroupplot[xlabel=$r$,
	ylabel=$\dfrac{\|X-\hat{X}\|_F}{\|X\|_F}$,
	xmin = 1,
	grid = both,
	xmax = 40,
	xtick = {1,10,20,...,40},
	ytick = {0,0.1,0.2,0.3,0.4,0.5,0.6},
	scaled ticks=false, 
	tick label style={/pgf/number format/fixed}]
	\addplot[color = FigColor1,mark = none] file{compl_MSD_cov_error_rast.txt}; 
	\label{line:error_MSD_r*};
	\addplot[color = FigColor2,mark = none, dashed] file{compl_MSD_cov_error_max_rast.txt}; 
	\label{line:error_MSD_max_r*};
	\nextgroupplot[xlabel=$r$,
	ylabel=$\rk$,
	xmin = 1,
	grid = both,
	xmax = 40,
	ymax = 40,
	ymin = 1,
	xtick = {1,10,20,...,40},
	ytick = {1,10,20,30,40},
	scaled ticks=false, 
	tick label style={/pgf/number format/fixed}]
	\addplot[color = FigColor1,mark = none] file{rank_MSD_fro.txt}; 
	\label{line:rank_MSD_r*};
	\addplot[color = FigColor2,mark = none, dashed] file{rank_MSD_spec.txt}; 
	\label{line:rank_MSD_max_r*};
	\end{groupplot}
	\node[text width=5cm ,align=left,anchor=north] at ([yshift=-7mm]my plots c1r1.south) {\subcaption{Relative completion errors of: \newline
			\ref{line:error_MSD_r*} \cref{opt:compl_cov_r} with $\|\cdot\|_{g,r*}=\|\cdot\|_{\ell_2,r\ast}$\newline
			\ref{line:error_MSD_max_r*} \cref{opt:compl_cov_r} with $\|\cdot\|_{g,r*}=\normkyrast{\cdot}$
			\label{fig.err_rpath_MSD}}};%
	\node[text width=5cm ,align=left,anchor=north] at ([yshift=-7mm]my plots c2r1.south) 
	{\subcaption{Rank of the solutions to: \newline 
			\ref{line:rank_MSD_r*} \cref{opt:compl_cov_r} with $\|\cdot\|_{g,r*}=\|\cdot\|_{\ell_2,r\ast}$\newline
			\ref{line:rank_MSD_max_r*} \cref{opt:compl_cov_r} with $\|\cdot\|_{g,r*}=\normkyrast{\cdot}$
			\label{fig.err_gpath_MSD}}};%
	\end{tikzpicture}
	\caption{Relative errors and ranks of solutions to \cref{opt:compl_cov_r} with $\|\cdot\|_{g,r\ast}=\|\cdot\|_{\ell_2,r\ast}$ and $\|\cdot\|_{g,r*}=\normkyrast{\cdot}$.}
	\label{fig.err_rpath_gpath_MSD}
\end{figure}

\Cref{fig.err_rpath_gpath_MSD} displays the relative errors and the ranks of the estimates obtained by \cref{opt:compl_cov_r} for different low-rank inducing norms as functions of $r$.
The nuclear norm minimization ($r=1$), as visualized in~Figures~\ref{fig.Xpp_nuc}~and~\ref{fig.Xvv_nuc}, gives the same rank as both the low-rank inducing Frobenius- and spectral norms for $r=2$. However, the latter approaches give better completions. The low-rank inducing spectral norm outperforms the low-rank inducing Frobenius norm for all $r \geq 2$. In particular, $r = 9$ gives the best completion, with a solution of rank 10 (see~Figures~\ref{fig.Xpp_max_r9}~and~\ref{fig.Xvv_max_r9}). It is interesting that the solutions to \cref{opt:compl_cov_r} with $r = 10$ for both the low-rank inducing Frobenius- and spectral norms are of rank 10. By \cref{prop:opt_reg}, there are no better feasible rank-10 solutions that minimize the Frobenius- and spectral norms, respectively. The solution to \cref{opt:compl_cov_r} with the low-rank inducing Frobenius norm and $r=10$, is shown in~Figure~\ref{fig.Xpp_ropt} and \ref{fig.Xvv_ropt}. The solution to the low-rank inducing spectral norm with  $r=10$ looks identical to Figures~\ref{fig.Xpp_max_r9}~and~\ref{fig.Xvv_max_r9}. 
\begin{figure}[pbht]
	\begin{center}
		\def\factor{.22}%
		\begin{tikzpicture}
		\begin{groupplot}[xmin = 1,xmax =20,
		ymin = 1,ymax =20,
		ytick = {16,11,6,1},
		yticklabels = {5,10,15,20},
		scaled ticks=false, 
		colormap/jet,
		every colorbar/.append style={
			yticklabel style={
				text width=1.75 em,
				align=left,
				/pgf/number format/.cd,
				fixed,
				fixed zerofill,
				precision=2}},
		group style={
			group name=left plots,
			group size=2 by 4,
			horizontal sep= 2.7cm,
			vertical sep= 1.26cm,
			group name=my plots},
		width=\factor \textwidth,
		height=\factor \textwidth,
		scale only axis
		]
		\nextgroupplot[axis equal image,colorbar,colorbar style ={ytick ={.5,1,1.5,2},yticklabel style={
				/pgf/number format/.cd,
				fixed,
				fixed zerofill,
				precision=1}},point meta min=0.0238,
		point meta max=2.4747]
		\addplot[thick,blue] graphics[xmin=1,ymin=1,xmax=20,ymax=20] {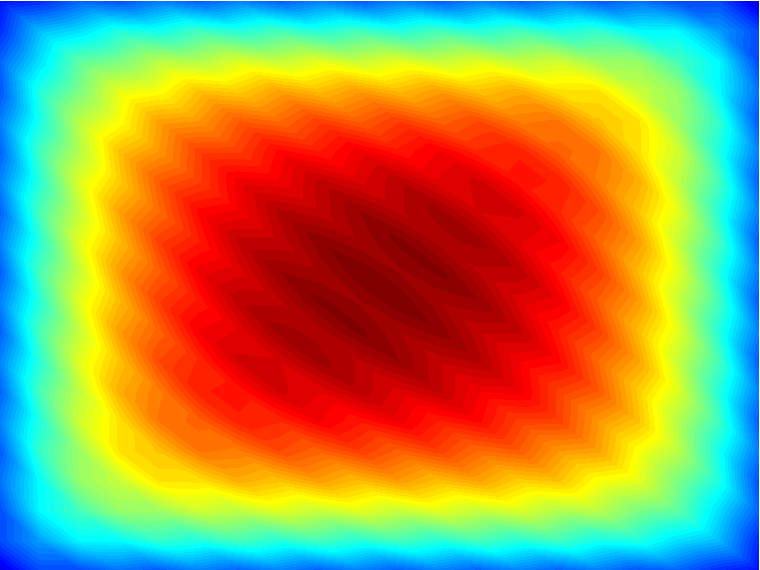};
		\coordinate (nl) at (current axis.north);
		\nextgroupplot[axis equal image,colorbar,colorbar style ={ytick ={-.1,-.05,0,.05,.1}},point meta min=-0.1260,
		point meta max= 0.1443]
		\addplot[thick,blue] graphics[xmin=1,ymin=1,xmax=20,ymax=20] {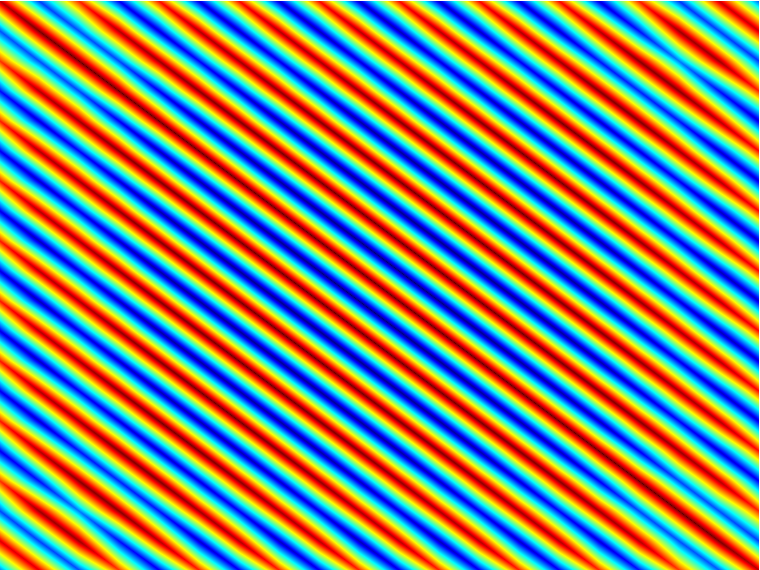};
		
		\nextgroupplot[axis equal image,colorbar,colorbar style ={ytick ={.5,1,1.5,2},yticklabel style={
				/pgf/number format/.cd,
				fixed,
				fixed zerofill,
				precision=1}},point meta min=0.0238,
		point meta max=2.4747]
		\addplot[thick,blue] graphics[xmin=1,ymin=1,xmax=20,ymax=20] {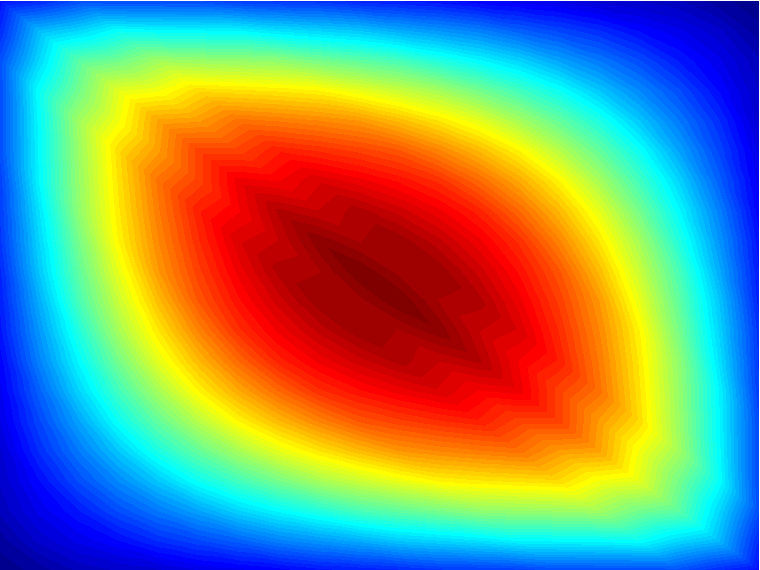};
		\coordinate (nl) at (current axis.north);
		\nextgroupplot[axis equal image,colorbar,colorbar style ={ytick ={-.1,-.05,0,.05,.1}},point meta min=-0.1260,
		point meta max= 0.1443]
		\addplot[thick,blue] graphics[xmin=1,ymin=1,xmax=20,ymax=20] {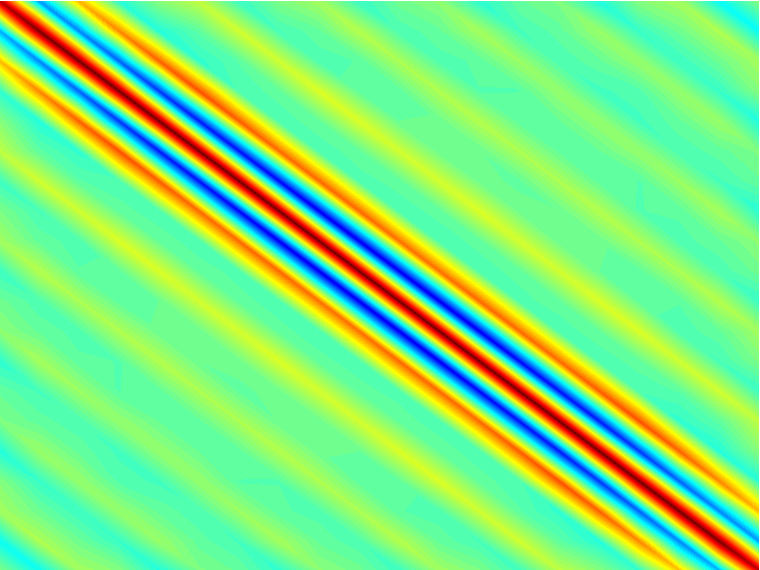};
		
		\nextgroupplot[axis equal image,colorbar,colorbar style ={ytick ={.5,1,1.5,2},yticklabel style={
				/pgf/number format/.cd,
				fixed,
				fixed zerofill,
				precision=1}},point meta min=0.0238,
		point meta max=2.4747]
		\addplot[thick,blue] graphics[xmin=1,ymin=1,xmax=20,ymax=20] {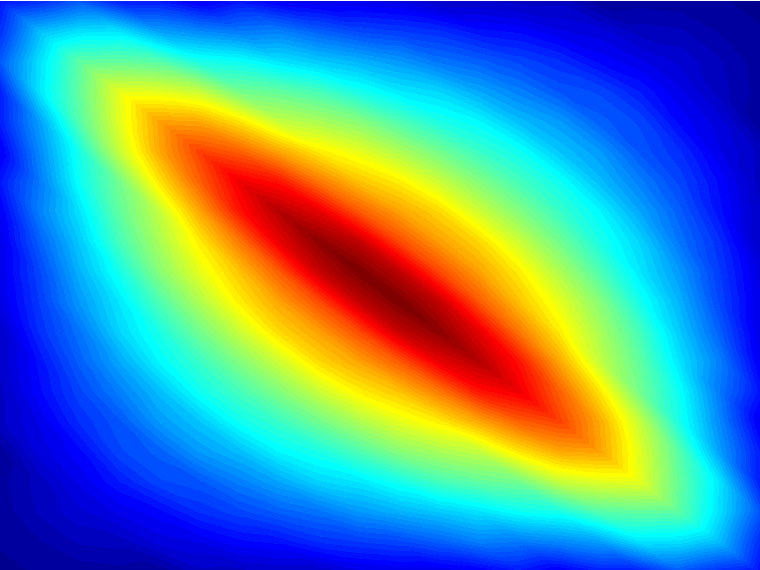};
		\coordinate (nl) at (current axis.north);
		\nextgroupplot[axis equal image,colorbar,colorbar style ={ytick ={-.1,-.05,0,.05,.1}},point meta min=-0.1260,
		point meta max= 0.1443]
		\addplot[thick,blue] graphics[xmin=1,ymin=1,xmax=20,ymax=20] {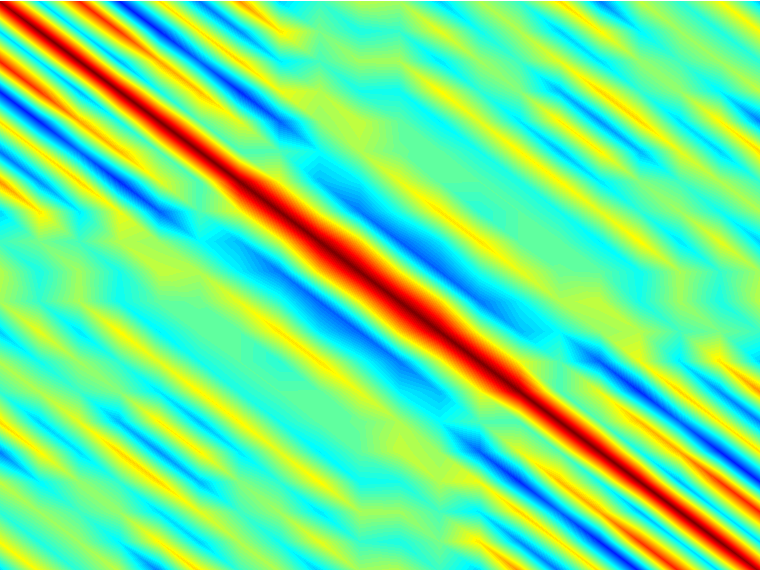};
		
		
		\end{groupplot}
		
		\node[text width=5cm ,align=center,anchor=north] at ([yshift=-2mm]my plots c1r1.south) {\subcaption{$\|\cdot\|_{g,r*}=\|\cdot\|_{\ell_1}$ in \cref{opt:compl_cov_r} \label{fig.Xpp_nuc}}};%
		\node[text width=5cm,align=center,anchor=north] at ([yshift=-2mm]my plots c2r1.south) {\subcaption{$\|\cdot\|_{g,r*}=\|\cdot\|_{\ell_1}$ in \cref{opt:compl_cov_r}\label{fig.Xvv_nuc}}};%
		
		\node[text width=5cm,align=center,anchor=north] at ([yshift=-2mm]my plots c1r2.south) {\subcaption{$\|\cdot\|_{g,r*}=\|\cdot\|_{10*}$ in \cref{opt:compl_cov_r}\label{fig.Xpp_ropt}}};%
		\node[text width=5cm,align=center,anchor=north] at ([yshift=-2mm]my plots c2r2.south) {\subcaption{$\|\cdot\|_{g,r*}=\|\cdot\|_{10*}$ in \cref{opt:compl_cov_r} \label{fig.Xvv_ropt}}};%
		\node[text width=5cm ,align=center,anchor=north] at ([yshift=-2mm]my plots c1r3.south) {\subcaption{$\|\cdot\|_{g,r*}=\|\cdot\|_{\ell_{\infty},9*}$ in \cref{opt:compl_cov_r}\label{fig.Xpp_max_r9}}};%
		\node[text width=5cm,align=center,anchor=north] at ([yshift=-2mm]my plots c2r3.south) {\subcaption{$\|\cdot\|_{g,r*}=\|\cdot\|_{\ell_{\infty},9*}$ in \cref{opt:compl_cov_r} \label{fig.Xvv_max_r9}}};%
		
		\node[text width=2cm,align=center,anchor=south, draw = none] at ([yshift= 0mm] my plots c1r1.north){};%

		\node[text width=5cm,align=center,anchor=north,above] at ([yshift=0mm]my plots c1r1.north) {$\hat{X}_{pp}$};
		\node[text width=5cm,align=center,anchor=north,above] at ([yshift=0mm]my plots c2r1.north) {$\hat{X}_{vv}$};
		\end{tikzpicture}%
		
	\end{center}
	\caption{Recovered covariance matrices of positions ($\hat{X}_{pp}$ to the left), and velocities ($\hat{X}_{vv}$ to the right), in the MSD system with $n=20$ masses resulting from problem~\cref{opt:compl_cov_r}, with different low-rank inducing norms.}\label{fig.MSDcompletion}
\end{figure} 

{{Similar to \Cref{fig:err_rastbest,fig:err_specbest}, it can been seen in \Cref{fig.err_rpath_MSD} that the preferred norm (here the spectral norm) significantly outperforms the other norm, even after dropping the rank constraint in \cref{prob:compl_cov_g}. This reveals the importance of having a variety of norm to choose from as well as having their low-rank inducing counterparts to obtain low rank. It lies outside of the scope of this paper to explain the physical relationship that motivates the choice of the spectral norm.}}

%% file: sec_conclusion.txt
\section{Conclusion}
\label{sec:conclusion}
We have proposed a family of low-rank inducing norms and regularizers. These norms are interpreted as the largest convex minorizers of a unitarily invariant norm that is restricted to matrices of at most rank $r$. One feature of these norms is that optimality interpretations in the form of a posteriori guarantees can be provided. In particular, it can be checked if the solutions to a convex relaxation involving low-rank inducing norms, also solve an underlying rank constrained problem. Our numerical examples indicate that this is useful for, e.g. the so-called matrix completion problem. A suitably chosen low-rank inducing norm yields significantly better completion and/or lower rank than the commonly used nuclear norm approach. This has been demonstrated on the basis of what we called low-rank inducing Frobenius- and spectral norms. Both norms have cheaply computable proximal mappings and are shown to have simple SDP representations. The class of low-rank inducing norms can be further broadened by using real-valued $r \geq 1$, {{ in which case
\begin{align*}
\| Y \|_{g^D,r} &:= g^D(\sigma_1(Y),\dots,\sigma_{\lfloor r\rfloor}(Y), (r - \lfloor r\rfloor) \sigma_{\lfloor r\rfloor + 1}(Y) ),
\end{align*}
where $\lfloor r\rfloor$ denotes the floor function. Here $r$ can be considered a regularization parameter.}} 

Finally, note that \cref{eq:low_rank_prob_main} may also be approached through so-called non-convex splitting methods as discussed in~\cite{grussler2017local}. There, the proximal mapping of $f(\normAast{\cdot}{g,r})$ is replaced by the proximal mapping of $f(\|\cdot\|_g)+I_{\rk(\cdot) \leq r}$. The latter can be computed whenever the proximal mappings of either $f(\normAast{\cdot}{g,r})$ or $f(\normA{\cdot}{g})$ are known. 

%% file: sec_appendix_v2.txt
\appendix
	\counterwithin{prop}{section}
	\counterwithin{defn}{section}
	\counterwithin{cor}{section}
	\counterwithin{lem}{section}
	\counterwithin{algorithm}{section}
	\setcounter{prop}{0}
	\setcounter{defn}{0}
	\setcounter{lem}{0}
	\setcounter{cor}{0}
	\setcounter{algorithm}{0}

		\section{Proofs to Results in \Cref{sec:norms}}
		\label{subsubsec:rast-norm}

\subsection{Proof to \cref{lem:norm}}
\label{app:lem_norm_pf}
	\begin{proof}
		Let $1\leq r \leq q: = \minmn $, $g: \mathbb{R}^q \to \mathbb{R}_{\geq 0}$ be a symmetric gauge function, $\Sigma_j(M) := \diag(\sigma_1(M),\dots,\sigma_{j}(M))$ for $M \in \Rmn$, and $1\leq j \leq q$. Then for all \linebreak $Y \in \Rmn$ it holds that
		\begin{align*}
		\| Y \|_{g^D,r} &= \max_{\stackrel{\rk(M) \leq r}{\|M\|_g \leq 1}} \langle M, Y \rangle = \max_{\stackrel{\rk(\Sigma_q(M)) \leq r}{\|\Sigma_q(M)\|_g \leq 1}} \langle \Sigma_q(Y),\Sigma_q(M) \rangle\\
		 &= \max_{\|\Sigma_r(M)\|_{g} \leq 1} \langle \Sigma_r(Y),\Sigma_r(M) \rangle = \|\Sigma_r(Y)\|_{g^D},
		\end{align*}
		where the second equality follows by \cite[Corollary~7.4.1.3(c)]{horn2012matrix}. Further, $\| \cdot \|_{g^D,r}$ is unitarily invariant, since $$\|\Sigma_r(Y)\|_{g^D} = g^D(\sigma_1(Y),\dots,\sigma_r(Y))$$ defines a symmetric gauge function (see~\cref{thm:symmgauge}). Similarly to the above, this implies that
		\begin{align*}
		\| M \|_{g,r\ast}&=\max_{\|Y\|_{g^D,r} \leq 1} \langle M,Y\rangle= \max_{g^D(\sigma_1(Y),\dots,\sigma_r(Y)) \leq 1} \sum_{i=1}^{q} \sigma_i(M) \sigma_i(Y)\\
		&= \max_{g^D(\sigma_1(Y),\dots,\sigma_r(Y)) \leq 1} \left[ \sum_{i=1}^{r} \sigma_i(M) \sigma_i(Y) + \sigma_r(Y) \sum_{i=r+1}^{q} \sigma_i(M) \right].
		\end{align*}
It remains to prove \cref{eq:norm_ineq} and \cref{eq:rank_norm}. The constraint set for $r+1$ is a superset of the constraint set for $r$ and by the definition of $\|\cdot\|_{g^D,r}$ in~\cref{eq:firstass} it follows that $\|Y\|_{g^D,r}\leq\|Y\|_{g^D,r+1}$. Therefore,
\begin{align*} 
\|M\|_{g,r\ast} &= \max_{\|Y\|_{g^D,r}\leq 1} \langle M,Y\rangle\geq\max_{\|Y\|_{g^D,r+1}\leq 1} \langle M,Y\rangle=\|M\|_{g,(r+1)*}.
	\end{align*} 
Note that $\|\cdot\|_{g^D} = \|\cdot\|_{g^D,q}$, which implies that $\|\cdot\|_{g,q*}=\|\cdot\|_g$ and thus \cref{eq:norm_ineq} is proven. The implication in \cref{eq:rank_norm} follows from the derived expression for $\|\cdot\|_{g,r*}$, since for rank-$r$ matrices $M$, $\sigma_i(M)=0$ for all $i\in\{r+1,\ldots,q\}$.
\end{proof}

\subsection{Proof to \cref{prop:nuclear_norm_characterization}}
 \label{app:prop_nuclear_norm_characterization_pf}
Using \cite[Corollary~7.4.1.3(c)]{horn2012matrix} it is possible to see that $g^D(\sigma_1) = |\sigma_1|$ if and only if $g(\sigma_1) = |\sigma_1|$. Thus, \cref{eq:grast_explic} yields for all $M \in \Rmn$  that
\begin{align*}
\|M\|_{g,1*} = \max_{\sigma_1(Y) \leq 1} \sigma_1(Y) \sum_{i=1}^{\minmn} \sigma_i(M) =\|M\|_{\ell_{\infty}^D} = \|M\|_{\ell_1},
\end{align*}
where we use the fact that the dual norm of the spectral norm is the nuclear norm (see~\cite[Theorem~5.6.42]{horn2012matrix}).		
\subsection{Proof to \cref{lem:convhull}}
\label{app:lem_convhull_pf}
\begin{proof}
By definition of $\|\cdot\|_{g^D,r}$ in \cref{eq:firstass} in \Cref{lem:norm}, it holds that for all $Y\in \Rmn$
\begin{align*}
\max_{X\in\conv(E_{g,r})}\langle X,Y\rangle=\max_{\stackrel{\rk(X)\leq r}{\|X\|_{g^D}\leq 1}}\langle X,Y\rangle=\|Y\|_{g^D,r} =\max_{\|X\|_{g,r*}\leq 1}\langle X,Y\rangle=\max_{X\in B_{g,r\ast}^{1}}\langle X,Y\rangle.
\end{align*}
Since $\conv(E_{g,r})$ and $B_{g,r*}^1$ are closed convex sets, this equality can only be fulfilled if the sets are equal (see~\cite[Theorem~V.3.3.1]{hiriart2013convex}).

Next, we prove the decomposition. Since the decomposition trivially holds for $M =0$, we assume that $M \neq 0$ and define $\bar{M}:=\|M\|_{g,r*}^{-1}M$. Then $\bar{M}\in B_{g,r*}^1=\conv(E_{g,r})$ and therefore can be decomposed as
\begin{align*}
\textstyle \bar{M}=\sum_i\alpha_i\tilde{M}_i \quad \text{with} \quad \sum_i\alpha_i=1, \ \alpha_i\geq 0,
\end{align*}
where all $\tilde{M}_i$ satisfy 
\begin{align*}
\|\tilde{M}_i\|_g=\|\tilde{M}_i\|_{g,r*}=1 \quad \text{and} \quad \rk(\tilde{M}_i)\leq r,
\end{align*}
where the first equality is from \cref{eq:rank_norm} in \cref{lem:norm}. Defining $M_i:=\tilde{M}_i\|M\|_{g,r*}$ gives 
\begin{align*}
\textstyle M=\sum_{i}\alpha_iM_i \quad \text{with} \quad \rk(M_i)\leq r
\end{align*}
and $$\|M_i\|_{g}=\|M_i\|_{g,r*}=\|\|M\|_{g,r*}\tilde{M}_i\|_{g,r*}=\|M\|_{g,r*}\|\tilde{M}_i\|_{g,r*}=\|M\|_{g,r*}.$$ This concludes the proof.
\end{proof}

\subsection{Proof to \Cref{prp:rank_r_ind_str_conv}}
\label{app:prp_rank_r_ind_str_conv_pf}
\begin{proof}
Let $\bar{M}=\sum_{i}\alpha_i M_i$ with $M_i\in E_{g,r}$ and $\alpha_i\in(0,1)$, $\sum_{i}\alpha_i=1$ be a convex combination of points in $E_{g,r}$. Then, by assumption, $$\textstyle \|\bar{M}\|_g=\|\sum_{i}\alpha_i M_i\|_g<\sum_{i}\alpha_i\|M_i\|_g=\sum_{i}\alpha_i=1$$ and thus $\bar{M}\not\in E_{g,r}$. Since $\conv(E_{g,r})=B_{g,r*}^1$, this implies that $E_{g,r}$ is the set of extreme points of $B_{g,r*}^1$.
\end{proof}

\subsection{Proof to \Cref{prp:rank_r_ind_spectral}}
\label{app:prp_rank_r_ind_spectral}
\begin{proof}
Let us start by showing that $\conv(\mathcal{E}_r)=B_{\ell_{\infty},r*}^1$. Since $\| \cdot \|_{\ell_{1},r}$ and $\| \cdot \|_{\ell_{\infty},r}$ are dual norms to each other, it follows by \cref{lem:convhull} that 
\begin{align*}
\|Y\|_{\ell_{1},r} = \max_{X \in B_{\ell_{\infty},r*}^1}\langle X,Y\rangle
=\max_{\stackrel{\rk(X)=r}{1=\sigma_{1}(X)=\ldots=\sigma_{r}(X)}}\sum_{i=1}^r\sigma_i(X)\sigma_i(Y)=\max_{X\in \conv(\mathcal{E}_r)}\langle X,Y\rangle,
\end{align*}
where the last two equalities are a result of \cite[Corollary~7.4.1.3(c)]{horn2012matrix}. However, $\conv(\mathcal{E}_r)$ and $B_{\ell_{\infty},r*}^1$ are closed convex sets and therefore this equation can only hold if the sets are identical (see~\cite[Proposition~V.3.3.1]{hiriart2013convex}).

It remains to show that no point in $\mathcal{E}_r$ can be constructed as a convex combination of other points in $\mathcal{E}_r$. To this end, note that a necessary condition for $M\in\mathcal{E}_r$ is that $$\|M\|_F^2=\sum_{i=1}^{\minmn} \sigma_{i}^2(M)=\sum_{i=1}^r\sigma_{i}^2(M)=r.$$ Let $\bar{M}=\sum_{i}\alpha_iM_i$ be an arbitrary convex combination with $\alpha_i>0$ and $\sum_{i}\alpha_i=1$, of distinct points $M_i\in\mathcal{E}_r$. By the strict convexity of $\|\cdot\|_F^2$, it holds that 
\begin{align*}
\textstyle \|\bar{M}\|_F^2=\|\sum_{i}\alpha_iM_i\|_F^2<\sum_{i}\alpha_i\|M_i\|_F^2=r\sum_{i}\alpha_i=r.
\end{align*}
Hence, $\bar{M}\not\in\mathcal{E}_r$ and this concludes the proof.
\end{proof}